\documentclass{article}
\usepackage{amsmath,amssymb,amsthm,bm}
\usepackage[a4paper]{geometry}
\usepackage{algorithm,algorithmic}
\usepackage{booktabs}
\usepackage[colorlinks=false,pdfborder={0 0 0}]{hyperref}

\usepackage{xcolor}

\usepackage{authblk}
\usepackage{multirow}
\usepackage{nicematrix}

\DeclareMathOperator{\cond}{cond}

\DeclareMathOperator{\diag}{diag}
\DeclareMathOperator{\Diag}{Diag}

\DeclareMathOperator{\Span}{span}
\DeclareMathOperator{\trace}{trace}

\renewcommand{\Re}{\mathop{\mathrm{Re}}}
\renewcommand{\Im}{\mathop{\mathrm{Im}}}

\newcommand{\fro}{\mathsf F}

\newcommand*{\set}[1]{\lbrace#1\rbrace}
\newcommand*{\SET}[2]{\left\lbrace#1\colon #2\right\rbrace}

\newcommand*{\trans}{^{\mathsf T}}
\newcommand*{\herm}{^{\mathsf H}}

\newcommand*{\itrans}{^{-\mathsf T}}

\newcommand{\bmat}[1]{\begin{bmatrix}#1\end{bmatrix}}

\newcommand*{\conj}[1]{\bar{#1}}

\newcommand*{\mi}{\mathrm i}

\def\adots{\mathinner{\mkern2mu\raise1pt\hbox{.}\mkern2mu
    \raise4pt\hbox{.}\mkern2mu\raise7pt\hbox{.}\mkern1mu}}

\newcommand*{\tol}{\mathtt{tol}}
\newcommand*{\macheps}{\bm u}
\newcommand*{\residual}{\mathtt{res}}

\newcommand*{\maxres}{\mathtt{res}_{\max}}

\definecolor{bkgndcolor}{rgb}{0.78,0.93,0.8}

\graphicspath{{fig/}}

\newenvironment{keywords}{\medskip\textbf{Keywords:}}{}

\newenvironment{MSCcodes}{\medskip\textbf{AMS subject classifications (2020).}}{}

\newtheorem{theorem}{Theorem}

\title{A Structure-Preserving LOBPCG Algorithm for the
Bethe--Salpeter Eigenvalue Problem}
\author[1]{Xinyu Shan}
\author[1,2]{Meiyue Shao}
\affil[1]{School of Data Science, Fudan University, Shanghai 200433, China}
\affil[2]{MOE Key Laboratory for Computational Physical Sciences, Fudan
University, Shanghai 200433, China}
\date{\today}

\begin{document}

\maketitle

\begin{abstract}
The Bethe--Salpeter eigenvalue problem is a structured eigenvalue problem
arising in many-body physics.
In practice, a few of the smallest positive eigenvalues and the corresponding
eigenvectors need to be computed.
In principle, the LOBPCG algorithm can be applied to solve this eigenvalue
problem.
However, direct application of the existing LOBPCG algorithm does not
utilize the inherent structure of the problem.
We design a structure-preserving eigensolver based on the indefinite LOBPCG
algorithm to efficiently solve the Bethe--Salpeter eigenvalue problem.
We propose an improved Hetmaniuk--Lehoucq trick for the indefinite inner
product, as well as an adaptive, multi-level orthogonalization strategy to
ensure the numerical stability of our algorithm.
Numerical experiments demonstrate that the proposed algorithm can efficiently
and accurately compute the desired eigenpairs.
Since the symplectic eigenvalue problem for symmetric positive definite
matrices can be transformed to the Bethe--Salpeter eigenvalue problem, our
algorithm can naturally be adopted as a symplectic eigensolver.
\end{abstract}



\section{Introduction}
\label{sec:introduction}
In the field of many-body physics, the two-particle Green's function is
governed by the Bethe--Salpeter equation (BSE)~\cite{SB1951}, which describes
electron--hole interaction effects.
The excitation energy levels, corresponding to the poles of the Green's
function, can be determined by computing the eigenvalues of a Hamiltonian
operator \(\mathcal{H}\).
After appropriate discretization, the Hamiltonian operator \(\mathcal{H}\) can
be discretized into a block matrix of the form
\begin{align}
\label{eq:BSE-matrix}
H=\bmat{A & B \\ -\conj B & -\conj A}\in\mathbb C^{2n\times2n},
\end{align}
where \(A\herm=A\) and \(B\trans=B\).
The Bethe--Salpeter Hamiltonian (BSH) matrix \(H\) can be expressed as the
product of two Hermitian matrices given by
\begin{align}
\label{eq:definition_COmega}
C_n=\bmat{I_n & 0 \\ 0 & -I_n},\qquad
\Omega=\bmat{A & B \\ \conj B & \conj A}.
\end{align}
In most physical systems, the matrix \(\Omega\) is positive definite.
In this case, \(H\) is referred to as a \emph{definite} BSH matrix, and its
eigenvalues are real and occur in positive and negative pairs.
In this paper, we restrict ourselves to definite BSH matrices unless otherwise
specified.

Several methods have been proposed to solve the Bethe--Salpeter eigenvalue
problem (BSEP).
One popular approach is the \emph{Tamm--Dancoff approximation}
(TDA)~\cite{RL2000}, which simplifies the problem by dropping the off-diagonal
blocks of \(H\), computing instead the eigenpairs of the simplified Hermitian
matrix.
However, the accuracy of TDA is sometimes terribly low so that researchers
become more and more interested in full BSE solvers~\cite{GMG2009,SdYDL2016}.
In~\cite{SdYDL2016,SY2017}, the authors established some basic theoretical
properties of the BSEP.
These properties are used to develop a structure-preserving parallel algorithm
for computing all eigenpairs of a definite BSH matrix~\cite{SdYDL2016}.
The \(\Gamma\)QR algorithm~\cite{GLZ2018} and a doubling
algorithm~\cite{GCL2019} are proposed to diagonalize a general (indefinite)
BSH matrix.
When only a few smallest positive eigenvalues are needed, there are also
iterative solvers for solving this problem~\cite{BDKK2017,GMG2009,GMG2011,
GLZ2018}.
In some practical applications, the optical absorption spectrum is of
interest.
Algorithms for this purpose have also been studied in~\cite{SdLYDL2018}.

In this paper, we develop a structure-preserving locally optimal block
preconditioned conjugate gradient (LOBPCG) algorithm to compute a few of the
smallest positive eigenvalues and their corresponding eigenvectors of a
definite BSH matrix.
In principle, the BSEP can be reformulated as a symmetric generalized
eigenvalue problem
\[
C_nZ=\Omega Z\Lambda^{-1},
\]
and thus can be solved by the existing LOBPCG algorithm~\cite{Knyazev2001}.
It is certainly possible to adjust the existing LOBPCG algorithm so that the
inherent structure of the BSH matrix is exploited.
However, to ensure numerical stability, a practical implementation of such an
algorithm has a relatively high computational cost on
(re-)orthogonalization~\cite{DSYG2018,HL2006}.
To enhance the computational efficiency, we focus on an equivalent
symmetric indefinite generalized eigenvalue problem
\[
\Omega Z=C_nZ\Lambda.
\]
We shall develop a structure-preserving LOBPCG algorithm based on an
indefinite variant of the LOBPCG algorithm~\cite{KMS2014} to solve this
problem.
The orthogonalization is based on the \(C_n\)-inner product, which is much
cheaper to evaluate compared to the \(\Omega\)-inner product.
The price to pay is that orthogonalization based on the \(C_n\)-inner product
can be numerically \emph{unstable} because the growth factor is theoretically
unbounded.
With the presence of rounding errors, the indefinite LOBPCG algorithm may
produce inaccurate solutions or even break down.
We shall discuss how to incorporate the improved Hetmaniuk--Lehoucq (IHL)
trick with reorthogonalization to enhance the numerical stability.
As a byproduct, our algorithm can also be used to solve the symplectic
eigenvalue problem, which is mathematically equivalent to the definite BSEP.

The rest of this paper is organized as follows.
In Section~\ref{sec:preliminaries}, we define the notation and introduce some
theoretical results relevant to the BSEP, along with the classical LOBPCG
algorithm.
Section~\ref{sec:algorithm} presents the implementation details of several
variants of the structure-preserving LOBPCG algorithm.
We provide the structured version of the IHL trick within the context of the
\(C_n\)-inner product.
In addition, we propose a multi-stage orthogonalization strategy and some
protection mechanisms to prevent interruption in the algorithm.
In Section~\ref{sec:symplectic}, the proposed structure-preserving LOBPCG
algorithm is applied to the symplectic eigenvalue problem based on the
equivalence theorem between the BSEP and the symplectic eigenproblem.
In Section~\ref{sec:experiments}, numerical experiments are performed to
demonstrate the effectiveness of the proposed algorithm.

\section{Preliminaries}
\label{sec:preliminaries}

\subsection{Bethe--Salpeter eigenvalue problem}
Let \(\Phi(X,Y)\) be a specific type of the structured matrix
\[
\Phi(X,Y)=\bmat{X & \bar Y \\ Y & \bar X},
\]
where the dimensions of matrices \(X\) and \(Y\) are same.
Then the matrix \(\Omega\) defined by \eqref{eq:definition_COmega} can be represented 
as \(\Omega=\Phi(A,\conj B)\).
Theorem~\ref{thm:BSE-eig} states that a definite BSH matrix has a structured
spectral decomposition.

\begin{theorem}[{\cite[Theorem~3]{SdYDL2016}}]
\label{thm:BSE-eig}
A definite BSH matrix \(H\) is diagonalizable and has a real spectrum.
Furthermore, it admits a spectral decomposition of the form
\begin{equation}
\label{eq:BSE-spectral-decomposition}
H=\bmat{X & \conj Y \\ Y & \conj X}
\bmat{\Lambda & 0 \\ 0 & -\Lambda}
\bmat{X & -\conj Y \\ -Y & \conj X}\herm
=\Phi(X,Y)C_n\Phi(\Lambda,0)\Phi(X,-Y)\herm,
\end{equation}
where \(\Lambda=\diag\set{\lambda_1,\ldots,\lambda_n}\) with
\(0<\lambda_1\leq\lambda_2\leq\dotsc\lambda_n\), and
\(\Phi(X,-Y)\herm\Phi(X,Y)=I_{2n}\).
\end{theorem}

We denote the set of \(2n\times2k\) \(C_n\)-orthonormal matrices as
\(\mathcal C(2n,2k)\), i.e.,
\[
\mathcal C(2n,2k)=\SET{Z\in\mathbb C^{2n\times2k}}{Z\herm C_nZ=C_k}.
\]
We further define \(\mathcal H(2n,2k)\) as the set of all \(2n\times 2k\)
matrices that have a structure akin to \(\Phi(U,V)\), specifically,
\[
\mathcal H(2n,2k)=\SET{Z\in\mathbb C^{2n\times2k}}{Z=\Phi(U,V)}.
\]
Then it follows from Theorem~\ref{thm:BSE-eig} that the eigenvectors of \(H\)
can be arranged as a matrix in \(\mathcal C(2n,2k)\cap\mathcal H(2n,2k)\).

Let
\begin{equation}
\label{eq:definition-M}
Q_{n}=\frac{1}{\sqrt2}\bmat{I_n & -\mi I_n \\ I_n & \mi I_n}
\in\mathbb C^{2n\times2n},
\qquad
M=\bmat{\Re{(A+B)} & \Im{(A-B)}\\ -\Im{(A+B)} & \Re{(A-B)}}
\in\mathbb R^{2n\times2n}.
\end{equation}
Then \(Q_n\) is unitary and \(M\) is real symmetric.
It is shown in~\cite{SdYDL2016} that the BSEP can be reduced to a real
Hamiltonian eigenvalue problem and vice versa.
In fact, it can be easily verified that
\begin{equation}
\label{eq:QM}
Q_{n}\herm C_nQ_{n}=-\mi J_n, \qquad Q_{n}\herm\Omega Q_{n}=M.
\end{equation}
Then \(Q_n\herm HQ_n=-\mi J_nM\), with \(J_nM\) being a real Hamiltonian
matrix.
Conversely, given a \(2n\times2n\) real symmetric matrix
\[
M=\bmat{M_{1,1} & M_{1,2} \\ M_{2,1} & M_{2,2}},
\]
there exists a BSH matrix of the form~\eqref{eq:BSE-matrix}, where \(A\) and
\(B\) are determined by
\begin{equation*}
\label{eq:constructAB}
A=\frac{M_{1,1}+M_{2,2}}{2}+\mi\cdot\frac{M_{1,2}-M_{2,1}}{2},
\qquad
B=\frac{M_{1,1}-M_{2,2}}{2}-\mi\cdot\frac{M_{1,2}+M_{2,1}}{2}.
\end{equation*}
We remark that the equivalence between the BSEP and the real Hamiltonian
eigenvalue problem does \emph{not} involve positive definiteness in general.
It follows from~\eqref{eq:QM} that the BSH matrix \(H\) is definite if and
only if \(M\) is positive definite.

\subsection{Symplectic eigenvalue problem}
A matrix \(S\in\mathbb{R}^{2n\times 2k}\) is called \emph{symplectic} if
\(S\trans J_nS=J_k\), where
\[
J_{n}=\bmat{0 & I_n \\ -I_n & 0}.
\]
Denote by \(\mathcal S(2n,2k)\) the set of \(2n\times 2k\) matrices with
symplectic columns, i.e.,
\[
\mathcal S(2n,2k)=\SET{S\in\mathbb R^{2n\times2k}}{S\trans J_nS=J_k}.
\]
A symplectic matrix is said to be \emph{orthosymplectic} if it is also an
orthogonal matrix.
The set of \(2k\times 2k\) orthosymplectic matrices is denoted by
\(\mathcal{OS}(2k)\).

Let \(M\in\mathbb{R}^{2n\times 2n}\) be a real symmetric positive definite
matrix.
Williamson's theorem (see Theorem~\ref{thm:Williamson}) states that \(M\) is
symplectically congruent to a diagonal matrix.

\begin{theorem}[\cite{Williamson1936}]
\label{thm:Williamson}
For any symmetric positive definite matrix \(M\in\mathbb{R}^{2n\times2n}\),
there exists a symplectic matrix \(S\in\mathbb{R}^{2n\times 2n}\) such that
\begin{equation}
\label{eq:Williamson}
S\trans MS=\bmat{\Lambda & \\& \Lambda},
\end{equation}
where \(\Lambda=\diag\set{\lambda_1,\ldots,\lambda_n}\) with
\(0<\lambda_1\leq\lambda_2\leq\cdots\lambda_n\).
\end{theorem}

The diagonal matrix \(\diag\set{\Lambda,\Lambda}\) in~\eqref{eq:Williamson} is
known as \emph{Williamson's normal form} of \(M\).
Let us partition \(S\) by column as \(S=[s_1,s_2,\dotsc,s_{2n}]\).
Then each \(\lambda_i\) is called a \emph{symplectic eigenvalue} of~\(M\),
with normalized \emph{symplectic eigenvectors} \(s_i\) and \(s_{i+n}\).

\subsection{The LOBPCG algorithm}
\label{subsec:LOBPCG}
The \emph{locally optimal block preconditioned conjugate gradient} (LOBPCG)
algorithm is a block eigensolver for solving standard or generalized symmetric
eigenvalue problems~\cite{Knyazev2001}.
Suppose that the \(k\) smallest eigenvalues of a Hermitian--definite pencil
\((A_1,A_2)\) (i.e., \(A_1\herm=A_1\in\mathbb C^{N\times N}\),
\(A_2\herm=A_2\in\mathbb C^{N\times N}\), and \(A_2\) is positive definite) are
of interest, where \(k\ll N\).
Mathematically, in the \(i\)th iteration of the LOBPCG algorithm, the
Rayleigh--Ritz procedure on the \(3k\)-dimensional search subspace
\(\Span\set{Z^{(i)},Z^{(i-1)},T^{-1} R^{(i)}}\) is performed,
where \(Z^{(i)}\in\mathbb C^{N\times k}\) consists of the approximate
eigenvectors in the \(i\)th iteration, \(R^{(i)}\in\mathbb C^{N\times k}\)
consists of the residuals, and \(T\) is the preconditioner.
The Ritz vectors corresponding to the \(k\) smallest Ritz vectors are chosen
as new approximate eigenvectors \(Z^{(i+1)}\).

In practice, the LOBPCG algorithm needs to be implemented \emph{very carefully}
in order to attain steady convergence and satisfactory
accuracy~\cite{DSYG2018,HL2006}.
Computing an orthonormal basis in the \(A_2\)-inner product is crucial to
maintain numerical stability.
The \emph{improved Hetmaniuk--Lehoucq (IHL) trick} proposed in~\cite{DSYG2018}
is a clever approach that can cheaply construct a matrix \(P^{(i)}\) such
that \(\bigl[Z^{(i)},P^{(i)}\bigr]\) forms an orthonormal basis of
\(\Span\set{Z^{(i)},Z^{(i-1)}}\).
Let \(\bigl[Z^{(i)},Z_{\perp}^{(i)}\bigr]\) be the orthonormal basis of the
search subspace \(\Span\set{Z^{(i)},Z^{(i-1)},T^{-1} R^{(i)}}\).
The eigenvectors of the Rayleigh--Ritz procedure are partitioned accordingly as
\[
V=\bmat{V_{1,1} & V_{1,2} \\ V_{2,1} & V_{2,2}},
\qquad (V_{1,1}\in\mathbb C^{k\times k},~V_{2,2}\in\mathbb C^{2k\times2k}).
\]
By computing the compact LQ factorization \(V_{1,2}=LQ\) (where
\(L\in\mathbb C^{k\times k}\) and \(Q\in\mathbb C^{k\times2k}\)), the IHL
trick selects
\[
\bigl[Z^{(i+1)},P^{(i+1)}\bigr]
=\bigl[Z^{(i)},Z_{\perp}^{(i)}\bigr]\cdot V\cdot\bmat{I_k \\ Q\herm}
\in\mathbb C^{N\times2k}
\]
as the new orthonormal basis of \(\Span\set{Z^{(i+1)},Z^{(i)}}\).%
\footnote{Even if \(\dim\bigl(\Span\set{Z^{(i+1)},Z^{(i)}}\bigr)<2k\), the IHL
trick still produces an orthonormal basis with \(2k\) vectors.}
This trick reduces the cost of orthogonalization of
\(\bigl[Z^{(i+1)},Z^{(i)}\bigr]\)
and enhances the numerical stability, as orthogonalization is only performed
on a small matrix.

For the definite BSEP, the most straightforward way to apply the LOBPCG
algorithm is to set \(A_1=C_n\) and \(A_2=\Omega\).
In~\cite{KMS2014}, the LOBPCG algorithm is extended to an indefinite setting,
requiring only a linear combination of \(A_1\) and \(A_2\) to be positive
definite.
This allows us to solve the problem with \(A_1=\Omega\) and \(A_2=C_n\).
The benefit of the indefinite setting is that orthogonalization with the
\(C_n\)-inner product is cheaper than that with the \(\Omega\)-inner product.
The price to pay is the risk of numerical instability.
We shall discuss how to develop an efficient and stable LOBPCG algorithm for
the BSEP in Section~\ref{sec:algorithm}.

\section{A structure-preserving LOBPCG algorithm}
\label{sec:algorithm}
The paper~\cite{KMS2014} presents a general framework of the indefinite LOBPCG
algorithm, and illustrates how to develop a structure-preserving
indefinite LOBPCG algorithm for the linear response eigenvalue problem.
As a generalization of the linear response eigenvalue problem, the BSEP, which
can be reformulated as \(\Omega Z =C_nZ\Lambda\), also fits the framework of
indefinite LOBPCG algorithm.
In the following we discuss how to exploit the structure of the BSEP to
develop an efficient and robust LOBPCG algorithm.

\subsection{Structured orthogonalization}
\label{subsec:orthogonalization}
Orthogonalization is a key component to maintain the numerical stability of the
LOBPCG algorithm~\cite{DSYG2018,HL2006}.
In order to develop an indefinite LOBPCG algorithm for
\(\Omega Z =C_nZ\Lambda\), we first discuss how to perform structured
orthogonalization in the \(C_n\)-inner product.

Suppose we have a structured basis for the search space of the LOBPCG
algorithm, denoted as \(U=\Phi(U_X,U_Y)\), where
\(U_X=[X,P_X,W_X]\in\mathbb{C}^{n\times 3k}\) and
\(U_Y=[Y,P_Y,W_Y]\in\mathbb{C}^{n\times 3k}\).
A natural requirement is that the structure of \(U\) is preserved after the
\(C_n\)-orthogonalization.

\subsubsection{A structured CGS procedure}
\label{subsec:CGS}
Consider the structured \(C_n\)-orthogonalization performed by the classical
Gram--Schmidt (CGS) procedure.
Rearrange the columns of \(U\) in the following structured block form
\[
\left[
\begin{array}{cc|cc|c|cc|cc|c|cc}
U_{X,1} & \conj{U}_{Y,1} &
U_{X,2} & \conj{U}_{Y,2} &
\cdots &
U_{X,p-1} & \conj{U}_{Y,p-1} &
U_{X,p} & \conj{U}_{Y,p} &
\cdots &
U_{X,3k} & \conj{U}_{Y,3k} \\
U_{Y,1} & \conj{U}_{X,1} &
U_{Y,2} & \conj{U}_{X,2} &
\cdots &
U_{Y,p-1} & \conj{U}_{X,p-1} &
U_{Y,p} & \conj{U}_{X,p} &
\cdots &
U_{Y,3k} & \conj{U}_{X,3k}
\end{array}
\right].
\]
Suppose that the first \(p-1\) blocks have already been orthogonalized to a
structured block form, \(\Phi(U_{X,1:p-1},U_{Y,1:p-1})\).
Orthogonalizing \([U_{X,p}\herm,U_{Y,p}\herm]\herm\) against
\(\Phi(U_{X,1:p-1},U_{Y,1:p-1})\) yields
\[
\bmat{U_{X,p}\\U_{Y,p}}
\gets\bigl(I-\Phi(U_{X,1:p-1},U_{Y,1:p-1})C_{p-1}
\Phi(U_{X,1:p-1},U_{Y,1:p-1})\herm C_n\bigr)
\bmat{U_{X,p}\\U_{Y,p}}.
\]
Then the updated \(p\)th block satisfies
\[
\bmat{U_{X,p} & \conj U_{Y,p} \\ U_{Y,p} & \conj U_{X,p}}\herm C_n
\bmat{U_{X,1:p-1} & \conj U_{Y,1:p-1} \\ U_{Y,1:p-1} & \conj U_{X,1:p-1}}=0.
\]
Moreover, the two columns in the \(p\)th block are automatically orthogonal to
each other in the \(C_n\)-inner product.
The structure of the \(p\)th block remains unaltered after normalization.
Therefore, we conclude that the matrix \(U\) after this CGS procedure
preserves the structure \(U=\Phi(U_X,U_Y)\), and satisfies
\begin{equation*}
\label{eq:ortho-Cn}
\Phi(U_X,U_Y)\herm C_n\Phi(U_X,U_Y)=C_{3k}.
\end{equation*}

In practice, inexact arithmetic often causes loss of orthogonality.
To alleviate this issue, it is recommended to perform orthogonalization twice.
A rounding error analysis in~\cite{ROS2015} shows that, under mild conditions
on \(U\in\mathbb C^{2n\times 2p}\), the \(C_n\)-orthonormal basis after CGS2,
denoted by \(U^{\{2\}}\), satisfies
\(\lVert(U^{\{2\}})\herm C_n U^{\{2\}}-C_p\rVert_2
=O(\macheps)\cdot\lVert U^{\{2\}}\rVert^2_2\).
Thus, reorthogonalization is helpful even for indefinite inner products.

We remark that the modified Gram--Schmidt (MGS) procedure, and its variant
with reorthogonalization (MGS2), can also preserve the block structure in the
context of \(C_n\)-inner product.
As it is straightforward to derive these algorithms, we do not discuss them
here.

\subsubsection{An indefinite SVQB algorithm}
\label{subsec:SVQB-like}
Alongside the CGS algorithm, the SVQB algorithm proposed by Stathopolous and
Wu~\cite{SW2002} can also be used to perform the \(C_n\)-orthogonalization.
We refer to it as the indefinite SVQB algorithm.
One of the advantages of the SVQB algorithm is that it performs all
operations through the matrix--matrix multiplication, thereby effectively
reducing communication costs.

Taking the structured matrix \(U=\Phi(U_X,U_Y)\in\mathbb{C}^{2n\times 2p}\) as
an example, we outline a simplified process of the indefinite SVQB algorithm.
Let \(M_U=U\herm C_nU\), and assume that \(M_U\) is nonsingular.
First, solve the eigenvalue problem
\begin{align}
\label{eq:ortho-eig}
M_UF=\Phi(U_X,U_Y)\herm C_n\Phi(U_X,U_Y)F=F\Sigma.
\end{align}
Although the matrix \(M_U\) is not a BSH matrix, its block structure is
similar to that in~\eqref{eq:BSE-matrix}.
In fact, \(M_U\) is a Hermitian matrix whose eigenvalues appear in pairs
\(\pm\Sigma_{+}\).
It can be shown that~\(M_U\) also has structured eigenvectors of the form
\(F=\Phi(F_X,F_Y)\) satisfying \(F\herm F=I\);
see~\cite{Shan2025} for details.
With the help of the structured spectral decomposition
\[
M_U=\Phi(F_X,F_Y)C_p\Phi(\Sigma{+},0)\Phi(F_X,F_Y)\herm,
\]
we then update \(\Phi(U_X, U_Y)\) as
\[
\Phi(U_X, U_Y)\gets\Phi(U_X,U_Y)\Phi(F_X,F_Y)\Phi(\Sigma_{+}^{-1/2},0).
\]
From \eqref{eq:ortho-eig}, we can infer that
\begin{equation*}
\label{eq:ortho-Cn}
\Phi(U_X,U_Y)\herm C_n\Phi(U_X,U_Y)=C_{p}.
\end{equation*}
We remark that in practice it is recommended to perform a diagonal scaling on
\(M_U\) (or, equivalently, normalize the columns of \(U\)) before computing
the spectral decomposition, because this preprocessing step can largely
enhance the numerical stability.

The authors in~\cite{SW2002} provided an error analysis and the loss of
orthogonality in the context of the standard inner product.
Using a similar trick, we provide a rough estimate for the loss of
\(C_n\)-orthogonality;
see Appendix~\ref{sec:appendix-SVQB}.
In general, we need to perform one step of reorthogonalization on the indefinite
SVQB algorithm to ensure, under mild assumptions, that the loss of
orthogonality of the new basis \(U^{\{2\}}\) satisfies
\(\lVert(U^{\{2\}})\herm C_n U^{\{2\}}-C_p\rVert_2
=O(\macheps)\cdot\lVert U^{\{2\}}\rVert^2_2\).

\subsubsection{Remedy on breakdown}
We remark that orthogonalization in the \(C_n\)-inner product has the risk of
serious breakdown due to normalizing nonzero \(C_n\)-neutral vectors.
By a \(C_n\)-neutral vector, we mean a vector \(q\) with \(q\herm C_nq=0\).
Although breakdown is uncommon in practice, once it indeed occurs (or near
\(C_n\)-neutral vectors are encountered), we suggest performing
orthogonalization in the \(\Omega\)-inner product as a remedy.
It is worth noting that in the context of \(\Omega\)-orthogonalization, the
two columns within a \(2n\times2\) block are \emph{not} automatically
\(\Omega\)-orthogonal to each other.
Therefore, it is necessary to perform an additional structure-preserving
\(\Omega\)-orthogonalization within such a block.

\subsection{Structured IHL trick}
\label{subsec:IHL}
In the LOBPCG algorithm, a reliable and efficient strategy for maintaining
numerical stability is to adopt an IHL trick~\cite{DSYG2018,HL2006} for
updating the basis.
In the following we discuss how this trick is implemented in the \(C_n\)-inner
product framework.

Suppose that the basis of the search subspace, \(U=\Phi(U_X,U_Y)\), satisfies
\(U\herm C_{n}U=C_{3k}\), where \(U_X=[X,P_X,W_X]\in\mathbb{C}^{n\times 3k}\)
and \(U_Y=[Y,P_Y,W_Y]\in\mathbb{C}^{n\times 3k}\).
We employ the structure-preserving algorithm in~\cite{SdYDL2016} to compute
eigenpairs of the small-sized BSEP in the Rayleigh--Ritz procedure, resulting
in
\begin{equation}
\label{eq:RR_BSEP}
\Phi(U_X,U_Y)\herm\Omega\Phi(U_X,U_Y)\Phi(V_X, V_Y)
=C_{3k}\Phi(V_X, V_Y)\bmat{\Theta_{+} & \\ & -\Theta_{+}},
\end{equation}
where \(\Phi(V_X,V_Y)\) satisfies that
\begin{equation}
\label{eq:V-ortho}
\Phi(V_X,V_Y)\herm C_{3k}\Phi(V_X,V_Y)=C_{3k}.
\end{equation}
Partition \(V_X\) and \(V_Y\) as follows:
\[
V_X=\bmat{V_{X,1} & V_{X,2}}=\bmat{V_{X,11} & V_{X,12}\\ V_{X,21} & V_{X,22}},\quad
V_Y=\bmat{V_{Y,1} & V_{Y,2}}=\bmat{V_{Y,11} & V_{Y,12}\\ V_{Y,21} & V_{Y,22}},
\]
where \(V_{X,1}\in\mathbb{C}^{3k\times k}\),
\(V_{Y,1}\in\mathbb{C}^{3k\times k}\),
\(V_{X,11}\in\mathbb{C}^{k\times k}\) and \(V_{Y,11}\in\mathbb{C}^{k\times k}\).
Rearrange the columns of \(V=\Phi(V_X,V_Y)\) in the form
\begin{align*}
\left[\begin{array}{c|c|c|c}
V_{X,1} & \conj V_{Y,1} & V_{X,2} & \conj{V}_{Y,2}\\
\hline
V_{Y,1} & \conj V_{X,1} & V_{Y,2} & \conj{V}_{X,2}\\ 
\end{array}\right]
=
\left[\begin{array}{c|c|c|c}
V_{X,11} & \conj{V}_{Y,11} & V_{X,12} & \conj{V}_{Y,12} \\ 
V_{X,21} & \conj{V}_{Y,21} & V_{X,22} & \conj{V}_{Y,22}\\
\hline
V_{Y,11} & \conj{V}_{X,11} & V_{Y,12} & \conj{V}_{X,12}\\ 
V_{Y,21} & \conj{V}_{X,21} & V_{Y,22} & \conj{V}_{X,22}
\end{array}\right],
\end{align*}
where \([V_{X,1}\herm,V_{Y,1}\herm]\herm\) and
\([\conj V_{Y,1}\herm,\conj V_{X,1}\herm]\herm\) are the eigenvectors
corresponding to the \(k\) smallest positive eigenvalues and the \(k\) largest
negative eigenvalues of \eqref{eq:RR_BSEP}, respectively.
Let
\[
\check V
=\bmat{\check V_{1,1} & \check V_{1,2} \\ \check V_{2,1} & \check V_{2,2}},
\]
where
\begin{align*}
\check V_{1,1}
&=\bmat{V_{X,11} & \conj{V}_{Y,11} \\ V_{Y,11} & \conj{V}_{X,11}},
&\check V_{1,2}
&=\bmat{V_{X,12} & \conj{V}_{Y,12} \\ V_{Y,12} & \conj{V}_{X,12}},\\
\check V_{2,1}
&=\bmat{V_{X,21} & \conj{V}_{Y,21} \\ V_{Y,21} & \conj{V}_{X,21}},
&\check V_{2,2}
&=\bmat{V_{X,22} & \conj{V}_{Y,22} \\ V_{Y,22} & \conj{V}_{X,22}}.
\end{align*}
According to \eqref{eq:V-ortho}, we have
\(\check V\herm \check C_{3k}\check V=\check C_{3k}\),
where \(\check C_{3k}=\diag\{C_k, C_{2k}\}\).
This can be reformulated to
\[
\bigg(\bmat{\check V_{1,1} \\ \check V_{2,1}}C_k\bmat{\check V_{1,1} 
\\ \check V_{2,1}}\herm +\bmat{\check V_{1,2} \\ \check V_{2,2}}C_{2k}\bmat{\check V_{1,2} 
\\ \check V_{2,2}}\herm\bigg)\check C_{3k}=I_{6k}.
\]
In order to orthogonalize \([0,\check V_{2,1}\herm]\herm\) against
\([\check V_{1,1}\herm,\check V_{2,1}\herm]\herm\), in theory we can perform
\[
\left(I_{6k}-\bmat{\check V_{1,1} \\ \check V_{2,1}}C_k\bmat{\check V_{1,1} \\
\check V_{2,1}}\herm \check C_{3k}\right)\bmat{0 \\ \check V_{2,1}}=
\bmat{\check V_{1,2} \\ \check V_{2,2}}C_{2k}\bmat{\check V_{1,2} \\ \check V_{2,2}}
\herm \check C_{3k}\bmat{0 \\ \check V_{2,1}}
=-\bmat{\check V_{1,2} \\ \check V_{2,2}}
C_{2k}\check V_{1,2}\herm C_k\check V_{1,1}.
\]
Because \([\check V_{1,2}\herm,\check V_{2,2}\herm]\herm\) is already
\(\check C_{3k}\)-orthonormal, in practice we only need to orthogonalize the
\(4k\times 2k\) matrix \(C_{2k}\check V_{1,2}\herm C_k\).
This matrix is also a structured one because it can be represented as
\[
C_{2k}\check V_{1,2}\herm C_{k}
=\bmat{V_{X,12}\herm & -V_{Y,12}\herm \\
-V_{Y,12}\trans & V_{X,12}\trans}
=\Phi(V_{X,12}\herm,-V_{Y,12}\trans).
\]
Performing the structured \(C_{2k}\)-orthogonalization on this matrix yields a
\(C_{2k}\)-orthonormal basis \(Q=\Phi(Q_X,Q_Y)\).
We then update
\begin{align}\label{eq:IHL_update}
\begin{aligned}
P&=\bmat{U_{X} & \conj{U}_{Y} \\ U_{Y} & \conj{U}_{X}}
\bmat{V_{X,2} & \conj V_{Y,2} \\ V_{Y,2} & \conj V_{X,2}}
\bmat{Q_{X} & \conj{Q}_{Y} \\ Q_{Y} & \conj{Q}_{X}}
=\bmat{P_{X} & \conj{P}_{Y} \\ P_{Y} & \conj{P}_{X}},\\
Z&=\bmat{U_{X} & \conj{U}_{Y} \\ U_{Y} & \conj{U}_{X}}
\bmat{V_{X,1} & \conj V_{Y,1} \\ V_{Y,1} & \conj V_{X,1}}
=\bmat{X & \conj{Y} \\ Y & \conj{X}}
\end{aligned}
\end{align}
so that \(P\herm C_nZ=0\), \(P\herm C_nP=C_k\), and \(Z\herm C_nZ=C_k\).

\subsection{A structure-preserving LOBPCG algorithm}
\label{subsec:ILOBPCG}

\subsubsection{A general framework of the indefinite LOBPCG algorithm}
In the following we develop a structure-preserving LOBPCG algorithm for the
BSEP.
We have seen that if we impose the orthogonal basis of the search subspace to
be of the form \(U=\Phi(U_X,U_Y)\), then the output of the Rayleigh--Ritz
procedure preserves this structure.
Naturally, the residuals also exhibit such a structure because
\[
\bmat{R_X \\ R_Y}
=\Omega\bmat{X \\ Y}-C_{n}\bmat{X \\ Y}\Theta_{+},
\qquad
\bmat{\conj R_Y \\ \conj R_X}
=\Omega\bmat{\conj Y \\ \conj X}-C_n\bmat{\conj Y \\ \conj X}(-\Theta_{+}).
\]
Furthermore, if the preconditioners \(T_+\) and \(T_-\) fulfill
\(T_{-}=C_nJ_n\conj T_{+}C_nJ_n\), the preconditioned residuals inherit the
same structure as
\[
\bmat{W_X \\ W_Y}=T_+\bmat{R_X \\ R_Y},
\qquad
\bmat{\conj W_Y \\ \conj W_X}=T_-\bmat{\conj R_Y \\ \conj R_X}.
\]
A general framework of the structure-preserving indefinite LOBPCG algorithm is shown in
Algorithm~\ref{alg:LOBPCG-general}.

\begin{algorithm}[!tb]
\caption{A general framework of the indefinite LOBPCG algorithm for the BSEP.}
\label{alg:LOBPCG-general}
\begin{algorithmic}[1]
\REQUIRE \(\Omega\in\mathbb{C}^{2n\times 2n}\):
         Hermitian positive definite matrix;\\
         \(T_{+}\in\mathbb{C}^{2n\times 2n}\):
         Hermitian positive definite preconditioner;\\
         \(l\in\mathbb{N}\): number of desired positive eigenvalues;\\
         \(X^{(0)}\), \(Y^{(0)}\in\mathbb{C}^{n\times k}\):
         Initial guess with \(k\geq l\).
\ENSURE  \(l\) smallest positive eigenpairs of \(C_n\Omega\).
\WHILE{not converged}
  \STATE Compute the residual \([R_X\herm,R_Y\herm]\herm\)
  \STATE Apply the preconditioner:
         \([W_X\herm,W_Y\herm]\herm\gets T_+[R_X\herm,R_Y\herm]\herm\)
  \STATE \(U_{X}\gets[X,P_{X},W_{X}]\),\quad
         \(U_{Y}\gets[Y,P_{Y},W_{Y}]\)
  \STATE \(C_n\)-orthogonalize \(\Phi(U_X,U_Y)\)
         \label{alg:step:C-ortho}
  \STATE Perform the Rayleigh--Ritz procedure to the pair \((\Omega,C_n)\) on
         \(\Span\{\Phi(U_{X},U_{Y})\}\)
  \STATE Update \(X\), \(Y\), \(P_X\), \(P_Y\) by selecting the \(k\) smallest
         positive Ritz values
         \label{alg:step:update}
\ENDWHILE
\STATE Return the desired eigenpairs
\end{algorithmic}
\end{algorithm}

\subsubsection{Algorithmic details on orthogonalization}
\label{subsubsec:orthogonalization}
In Step~\ref{alg:step:update} of Algorithm~\ref{alg:LOBPCG-general}, the
structured IHL trick is highly recommended.
The \(C_{2k}\)-orthonormal basis \(Q\) in the IHL trick can be constructed
by the indefinite SVQB algorithm.
Then~\eqref{eq:IHL_update} can be used to update the basis so that \([Z,P]\)
is \(C_n\)-orthonormal.
To obtain a \(C_n\)-orthonormal basis of \([Z,P,W]\), we need to perform a
two-stage orthogonalization on \(W\) in Step~\ref{alg:step:C-ortho}.
The matrix \(W\) is first orthogonalized against \([Z,P]\) using a block CGS
algorithm.
Then the indefinite SVQB algorithm with reorthogonalization is performed on \(W\) to
produce a \(C_n\)-orthonormal basis.
Sometimes this two-stage orthogonalization needs to be repeated once more to
enhance the orthogonality.

Unlike the IHL trick in a positive definite inner product, in the indefinite LOBPCG
algorithm there is an additional risk of losing the \(C_n\)-orthogonality of
\([\hat Z,\hat P]\) in the IHL trick due to the accumulation of rounding
errors;
see Appendix~\ref{sec:appendix-basis} for details.
To alleviate the impact of rounding errors, it is recommended to explicitly
reorthogonalize \([\hat Z,\hat P]\) after the IHL trick update.

In principle, the reorthogonalization of \([\hat Z,\hat P]\) can be
accomplished by any \(C_n\)-orthogonalization algorithm.
However, in practice the SVQB algorithm is \emph{not} recommended here.
The purpose of reorthogonalization is to improve the orthogonality of an
approximately orthogonal basis with
\([\hat Z,\hat P]\herm C_n[\hat Z,\hat P]\approx C_{2k}\).
In CGS/MGS, the output of reorthogonalization is close to the input, so that
the (nearly) converged Ritz vectors only have minor changes in the subsequent
Rayleigh--Ritz process.
However, in the SVQB algorithm, the output of reorthogonalization may be
far away from \([\hat Z,\hat P]\) because the eigenvectors of \([\hat Z,\hat P]
\herm C_n[\hat Z,\hat P]\) are not necessarily close to \(I_{4k}\).
This often leads to less accurate Ritz vectors in the Rayleigh--Ritz process
due to rounding errors.

In the initial stages of the indefinite LOBPCG algorithm, explicit
reorthogonalization of \([\hat Z,\hat P]\) can sometimes be safely skipped.
To reduce the computational overhead, we employ a selective and adaptive
reorthogonalization strategy.
Instead of carefully monitoring the loss of orthogonality,
we randomly select a trial vector \(g\in\mathbb{C}^{2k}\), and compute
\begin{align*}
E_1\cdot g&=[\hat Z_X,\hat P_X]\herm([\hat Z_X,\hat P_X]\cdot g)
-[\hat Z_Y,\hat P_Y]\herm([\hat Z_Y,\hat P_Y]\cdot g)-g,\\
E_2\cdot g&=[\hat Z_Y,\hat P_Y]\trans([\hat Z_X,\hat P_X]\cdot g)
-[\hat Z_X,\hat P_X]\trans([\hat Z_Y,\hat P_Y]\cdot g).
\end{align*}
An additional \(C_n\)-orthogonalization step is performed only if
\begin{equation}
\label{eq:selective-orthogonalization}
(\lVert E_1\cdot g\rVert_\infty+\lVert E_2\cdot g\rVert_\infty)
\geq\min\set{\tau_0,\lVert\residual\rVert\cdot10^{-1}},
\end{equation}
where \(\tau_0\) is a prescribed constant (e.g., \(\tau_0=O(\macheps^{1/2})\)),
and \(\lVert\residual\rVert\) is the residual norm of the desired eigenpairs 
at the current iteration.
This heuristic strategy avoids unnecessary reorthogonalization when the
accuracy of the approximate eigenpairs is relatively low.
We call this variant of the indefinite LOBPCG algorithm, equipped with the IHL
trick and the selective reorthogonalization
strategy~\eqref{eq:selective-orthogonalization}, the \emph{LOBPCG-CIHL}
algorithm.

\subsubsection{An adaptive structured LOBPCG algorithm}
\label{subsubsec:adaptive}
Our computational experiences suggest that the LOBPCG-CIHL algorithm works
well in most cases, although this is not theoretically guaranteed by the a
priori worst-case rounding error analysis.
When the convergence curve of the LOBPCG-CIHL algorithm starts to oscillate
due to rounding errors, the standard LOBPCG algorithm on the
Hermitian--definite pencil \((C_n,\Omega)\) can be used to refine the accuracy.
As mentioned in Section~\ref{subsec:orthogonalization}, the orthogonalization
in the \(\Omega\)-inner product, which is required in the standard LOBPCG
algorithm, can be performed in a structure-preserving manner.
The projected subproblem in the Rayleigh--Ritz procedure also possesses a
BSH-like structure (see~\eqref{eq:ortho-eig}), and can be solved by a
structured algorithm in~\cite{Shan2025}.
We refer to this structured LOBPCG algorithm with the IHL trick operating in
the \(\Omega\)-inner product as the \emph{LOBPCG-\(\varOmega\)IHL}
algorithm.

In practice, we prefer using the LOBPCG-CIHL algorithm whenever possible, and
switch to the more expensive LOBPCG-\(\Omega\)IHL algorithm only as a
safeguard.
We propose an adaptive LOBPCG algorithm as illustrated in
Algorithm~\ref{alg:LOBPCG-CIHL-OmegaIHL}.
A natural question is how to detect the convergence stagnation in the
LOBPCG-CIHL algorithm.
A simple heuristic strategy is to monitor the slope of the convergence curve.
Since the LOBPCG algorithm typically exhibits a linear convergence rate for
large-scale problems~\cite{BL2022,Knyazev2001,SL2025}, the (asymptotic)
convergence curve in the logarithmic scale follows a straight line.
When the convergence curve significantly deviates from the ideal straight line,
we can switch from LOBPCG-CIHL to LOBPCG-\(\Omega\)IHL.

\begin{algorithm}[tp!]
\caption{An adaptive LOBPCG algorithm for the BSEP}
\label{alg:LOBPCG-CIHL-OmegaIHL}
\begin{algorithmic}[1]
\REQUIRE \(\Omega\in\mathbb{C}^{2n\times 2n}\):
         Hermitian positive definite matrix;\\
         \(T_{+}\in\mathbb{C}^{2n\times 2n}\):
         Hermitian positive definite preconditioner;\\
         \(l\in\mathbb{N}\): number of desired positive eigenvalues;\\
         \(X^{(0)}\), \(Y^{(0)}\in\mathbb{C}^{n\times k}\):
         Initial guess with \(k\geq l\).
\ENSURE  \(l\) smallest positive eigenpairs of \(C_n\Omega\).
\STATE \(\texttt{Current\_Alg}\gets\text{LOBPCG-CIHL}\)
\WHILE{not converged}
  \IF{\(\texttt{Current\_Alg} = \text{LOBPCG-CIHL}\)}
    \STATE Perform one step of LOBPCG-CIHL and check convergence
    \IF{the convergence stagnates}
      \STATE \texttt{Current\_Alg} \(\gets\)  LOBPCG-\(\Omega\)IHL
    \ENDIF
  \ELSE
    \STATE Perform one step of LOBPCG-\(\Omega\)IHL and check convergence
  \ENDIF
\ENDWHILE
\end{algorithmic}
\end{algorithm}
\section{Application to the symplectic eigenvalue problem}
\label{sec:symplectic}
It is known that the symplectic eigenvalue problem is equivalent to the
definite Bethe--Salpeter eigenvalue problem~\cite{SZ2023}.
In the following we provide two detailed statements on the equivalence.

\begin{theorem}
\label{thm:equivalence}
Let \(H\) be a definite BSH matrix as defined in~\eqref{eq:BSE-matrix},
and \(M\) be a symmetric positive definite matrix
defined by~\eqref{eq:definition-M}.
Then Theorems~\ref{thm:BSE-eig} and~\ref{thm:Williamson} are equivalent to
each other.
Moreover, the matrices \(\Lambda\) in~\eqref{eq:BSE-spectral-decomposition}
and~\eqref{eq:Williamson} are identical.
\end{theorem}

\begin{proof}
Starting from Theorem~\ref{thm:Williamson}, we obtain
\[
Q_n\herm HQ_nS=-\mi J_nMS
=-\mi J_nS\itrans\bmat{\Lambda & 0 \\ 0 & \Lambda}
=-\mi SJ_n\bmat{\Lambda & 0 \\ 0 & \Lambda},
\]
where \(Q_n\) and \(M\) are defined in~\eqref{eq:definition-M}.
Let \(Z=Q_nSQ_n\herm\).
Then
\begin{multline*}
HZ=Q_n(Q_n\herm HQ_nS)Q_n\herm
=Q_n\left(-\mi SJ_n\bmat{\Lambda & 0 \\ 0 & \Lambda}\right)Q_n\herm\\
=(Q_nSQ_n\herm)(-\mi Q_nJ_nQ_n\herm)\bmat{\Lambda & 0 \\ 0 & \Lambda}
=ZC_n\bmat{\Lambda & 0 \\ 0 & \Lambda}.
\end{multline*}
Theorem~\ref{thm:BSE-eig} is thus valid.
As the proof above only involves unitary similarity, which is invertible, we
can also derive Theorem~\ref{thm:Williamson} from Theorem~\ref{thm:BSE-eig}.
\end{proof}

When only the \(k\) smallest positive eigenvalues of \(H\) are of interest,
two different forms of the trace minimization principle
(\cite[Theorem~4]{SY2017} and~\cite[Theorem~5]{BJ2015}), which are also
equivalent to each other, can be used to develop optimization-based
eigensolvers.
Theorem~\ref{thm:trace-min} characterizes the trace minimization principle.
The equivalence can be shown by the same technique as in the proof of
Theorem~\ref{thm:equivalence}, and is hence omitted.

\begin{theorem}
\label{thm:trace-min}
Let \(H\) be a definite BSH matrix as defined in~\eqref{eq:BSE-matrix},
and \(M\) be a symmetric positive definite matrix
defined by~\eqref{eq:definition-M}.
Then
\begin{align}
2\sum_{i=1}^k\lambda_i
&=\min_{\substack{S\in\mathbb C^{2n\times2k}\\S\herm J_nS=J_k}}
\trace(S\herm MS)
=\min_{\substack{S\in\mathcal S(2n,2k)}}\trace(S\trans MS)
\nonumber\\
&=\min_{\substack{Z\in\mathbb C^{2n\times2k}\\Z\herm C_nZ=C_k}}
\trace(Z\herm\Omega Z)
=\min_{\substack{Z\in\mathcal H(2n,2k)\cap\mathcal C(2n,2k)}}
\trace(Z\herm\Omega Z).\label{eq:trace-min}
\end{align}
\end{theorem}

With the help of~\eqref{eq:trace-min}, we automatically obtain a
structure-preserving LOBPCG algorithm that computes the \(k\) smallest
symplectic eigenvalues of a \(2n\times2n\) symmetric positive definite matrix.
A straightforward approach is to transform the positive definite matrix \(M\)
to the definite BSH matrix \(-\mi Q_nJ_nMQ_n\herm\) and then apply
Algorithm~\ref{alg:LOBPCG-CIHL-OmegaIHL}.

\section{Numerical experiments}
\label{sec:experiments}
In this section, we use experimental results to illustrate the effectiveness
and efficiency of our structure-preserving LOBPCG algorithm.
All numerical experiments were performed using MATLAB R2022b on a Linux server
with two 16-core Intel Xeon Gold~6226R~2.90~GHz CPUs and 1024~GB of main
memory.

For each test problem, we compute the \(l\) smallest positive eigenvalues and
the corresponding eigenvectors of \(H=C_n\Omega\) using the LOBPCG algorithm
with \(k=\max\set{\lceil3/2\cdot l\rceil,l+5}\).
The precision of the approximate eigenpair \((\theta_i,z_i)\) is measured
using the normalized residual
\[
\residual_i=\frac{\lVert\Omega z_i-C_nz_i\theta_i\rVert_2}
{(\lVert\Omega\rVert_2+\theta_i)\lVert z_i\rVert_2},
\]
where \(\lVert\Omega\rVert_2\) is estimated through
\(\lVert\Omega\rVert_2\approx\lVert\Omega G_r\rVert_{\fro}
/\lVert G_r\rVert_{\fro}\) using a Gaussian random matrix
\(G_r\in\mathbb C^{2n\times t}\) with \(t\ll n\).
The algorithm terminates if either
\[
\maxres=\max_{1\leq i\leq l}\residual_i\leq\tol=10^{-14},
\]
or the number of iterations exceeds \(\mathtt{max\_iter}=200\).

In Algorithm~\ref{alg:LOBPCG-CIHL-OmegaIHL}, we begin monitoring the slope
only after \(\maxres\) falls below \(10^{-10}\).
A switch occurs if the residual norm exhibits an upward trend or the
convergence curve significantly deviates from the expected linear behaviour in
the logarithmic scale.
Define the secant line slope between \((k-p)\)th and \(k\)th steps as
\[
s_{p}^{(k)}=\frac{\log_{10}(\maxres^{(k)})-\log_{10}(\maxres^{(k-p)})}{p}.
\]
Specifically, we switch from the \(C_n\)-inner product to the \(\Omega\)-inner
product if
\[
\maxres^{(k)}>\max\set{\maxres^{(k-1)},\maxres^{(k-2)}}
\quad\text{or}\quad
s_{5}^{(k)}>\frac{s_{10}^{(k)}}{2}.
\]

\subsection{Bethe--Salpater eigenvalue problems}
In this section, we examine several examples derived from the discretized
Bethe--Salpeter eigenvalue problems listed in Table~\ref{tab:BSE-example}.
The \(2n\times 2n\) dense BSH matrices are associated with the naphthalene,
gallium arsenide (\texttt{GaAs}), boron nitride (\texttt{BN}), and phosphorene
nanoribbon (\texttt{PNR}), respectively.
Preconditioners of all LOBPCG variants are set to \(\Phi(\Diag(A),0)\),
where \(\Diag(A)\) represents a diagonal matrix composed of the diagonal
elements of \(A\).

\begin{table}[!tb]
\centering
\caption{List of test examples of the BSE.}
\label{tab:BSE-example}
\begin{tabular}{ccccc}
\toprule
Cases & Name & Size (\(n\)) & \(\#\)Desired (\(l\))\\
\midrule
      1     &\texttt{naphthalene} & \phantom{00,0}32  & 3\\
      2     &\texttt{GaAs}        & \phantom{00,}128  & 12\\
      3     &\texttt{BN1}         & \phantom{0}2,304  & 23\\
      4     &\texttt{BN2}         & \phantom{0}2,304  & 50\\
      5     &\texttt{PNR}         & 10,000            & 50\\
\bottomrule
\end{tabular}
\end{table}

We evaluate the performance of several structured LOBPCG variants, including
LOBPCG-C (a simple ILOBPCG algorithm which uses CGS2 for \(C_n\)-orthogonalization 
without the IHL trick), LOBPCG-\(\Omega\)IHL (see Section~\ref{subsubsec:adaptive}),
LOBPCG-CIHL (see Section~\ref{subsubsec:orthogonalization}), and
Algorithm~\ref{alg:LOBPCG-CIHL-OmegaIHL}.
Convergence histories are illustrated in Figure~\ref{fig:compare_LOBPCG}.
For visual clarity, the oscillatory tails of LOBPCG-C are truncated in specific subplots.
The performance of LOBPCG variants is reported in Table~\ref{tab:performance-LOBPCG}.
When handling larger-scale problems or computing a lot of eigenvalues, LOBPCG-C
exhibits stagnation around \(10^{-12}\).
Although LOBPCG-CIHL may exhibit persistent residual oscillations, it ultimately 
reaches the desired accuracy level.
LOBPCG-\(\Omega\)IHL demonstrates superior stability at the cost of more computational 
overhead.
Algorithm~\ref{alg:LOBPCG-CIHL-OmegaIHL} effectively balances numerical accuracy and
efficiency.

\begin{figure}[!tb]
\centering
\begin{tabular}{cc}
\includegraphics[width=0.45\textwidth]{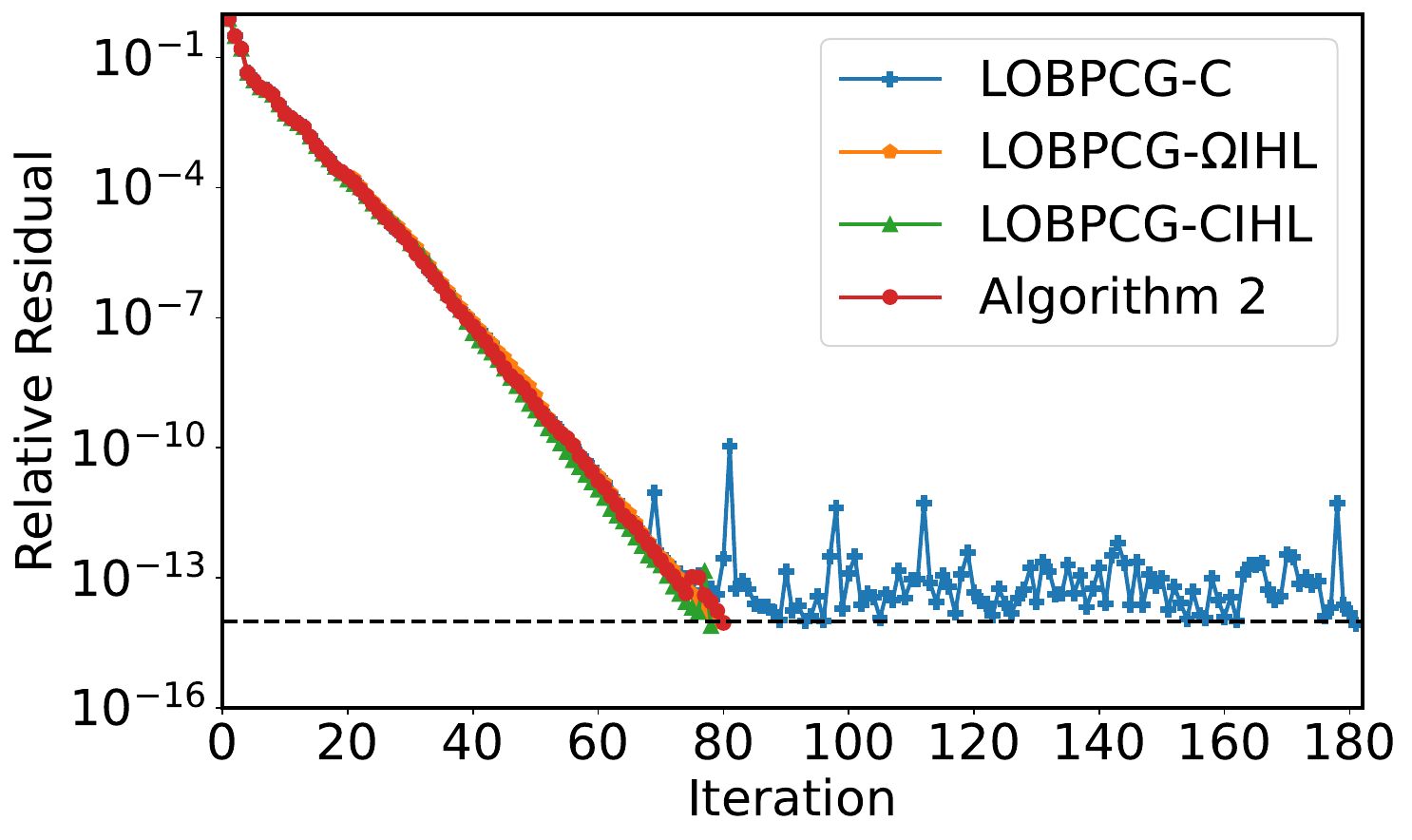} &
\includegraphics[width=0.45\textwidth]{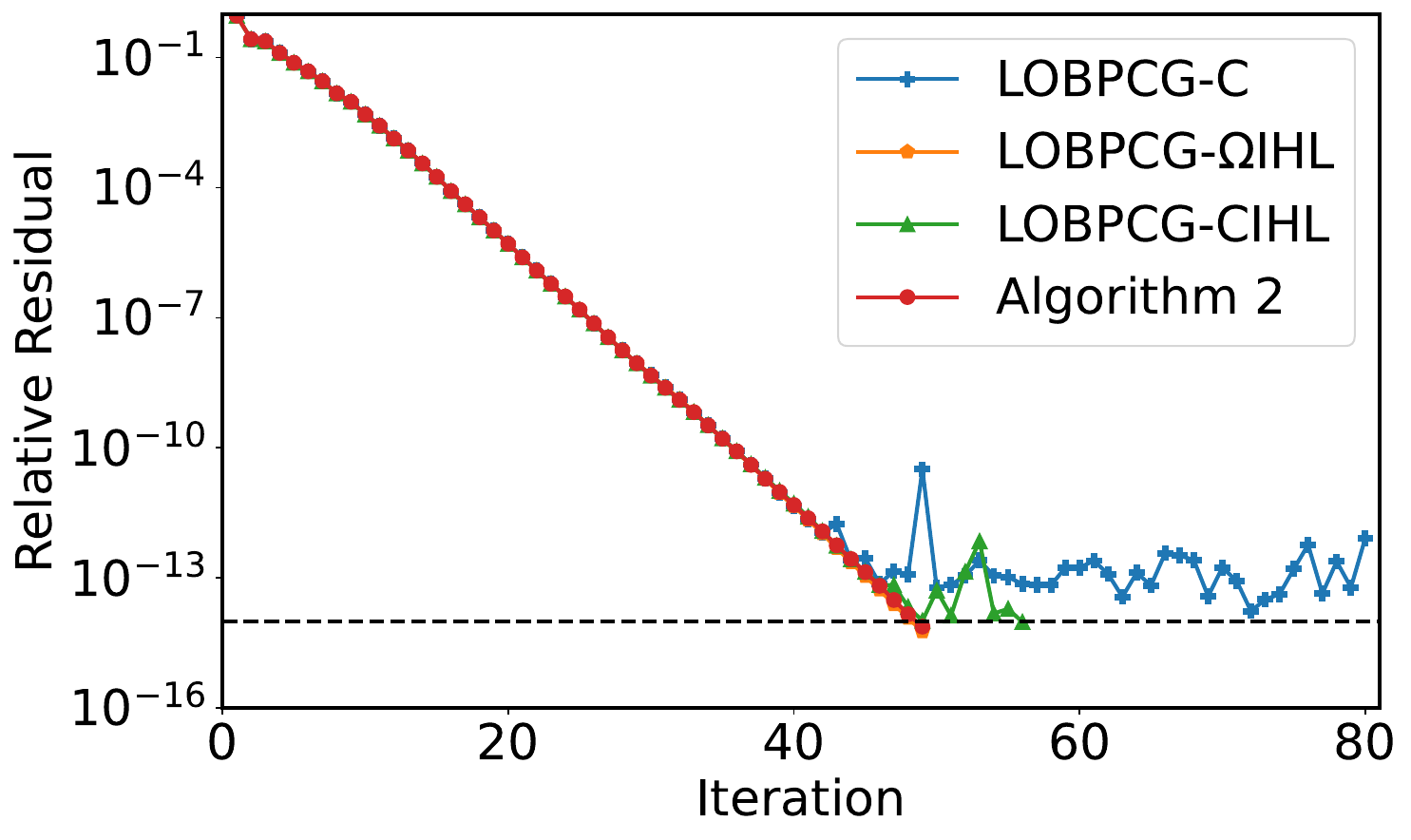} \\
Case \(2\) & Case \(3\)\\
\end{tabular}
~\\~\\~\\
\begin{tabular}{cc}
\includegraphics[width=0.45\textwidth]{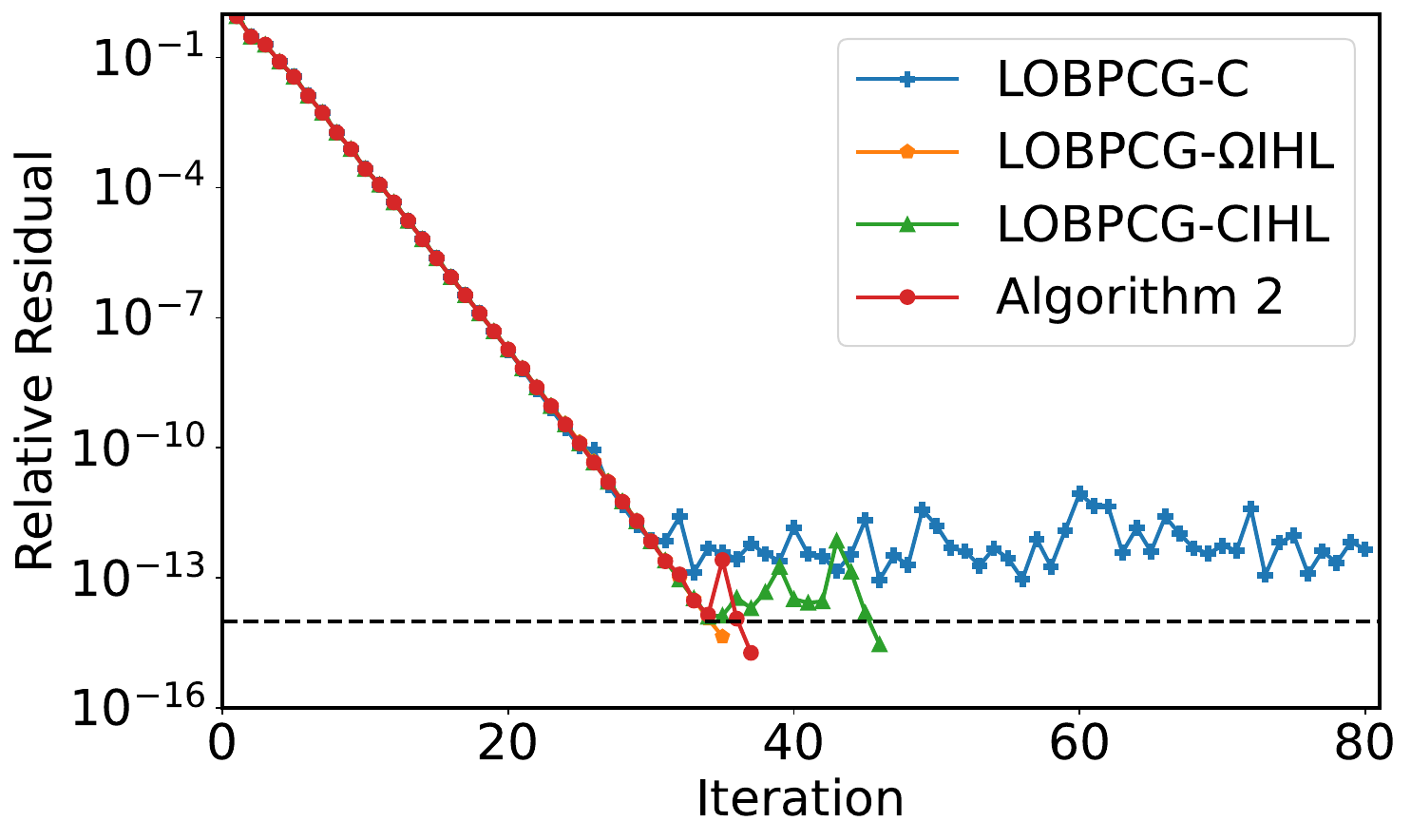} &
\includegraphics[width=0.45\textwidth]{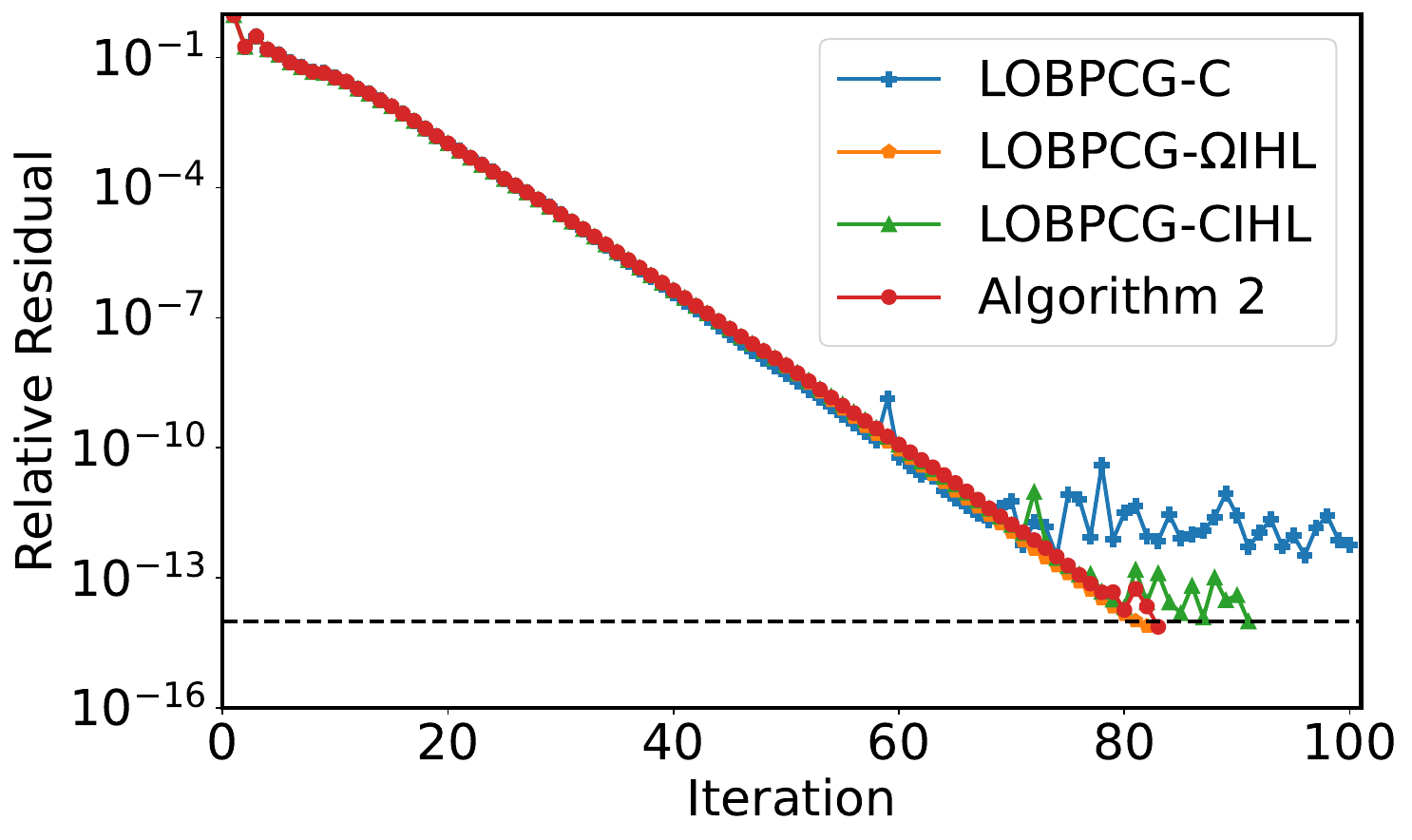} \\
Case \(4\) & Case \(5\)\\
\end{tabular}
\caption{Comparison of different LOBPCG variants.
Convergence history for the smallest example, \texttt{naphthalene} (Case~1),
is omitted because all variants follow nearly identical linear convergence.}
\label{fig:compare_LOBPCG}
\end{figure}

\begin{table}[!tb]
\centering
\small
\caption{Performance comparison of LOBPCG variants for test examples
listed in Table~\ref{tab:BSE-example}.
Data in boldface indicate that the desired level of accuracy is not achieved
within the maximum number of iterations.}
\label{tab:performance-LOBPCG}
\setlength{\extrarowheight}{1pt}
\begin{tabular}{cccccc}
\toprule
Cases &    Metric          & LOBPCG-C & LOBPCG-$\Omega$IHL & LOBPCG-CIHL 
                           & Algorithm~\ref{alg:LOBPCG-CIHL-OmegaIHL} \\
\hline
\multirow{3}*{1}     & Iteration & 63 & 62 & 63 & 65  \\
                     & Time (s)  & 0.2697 & 0.2381 & 0.2670 & 0.2472 \\
                     & Residual  & \(9.393\times 10^{-15}\) & \(8.163\times 10^{-15}\) 
                                 & \(7.944\times 10^{-15}\) & \(9.583\times 10^{-15}\) \\
\midrule 
\multirow{3}*{2}    & Iteration & 181 & 79 & 78 & 80  \\
                    & Time (s)  & 3.176 & 0.9522 & 1.070 & 1.102\\
                    & Residual  & \(8.466\times 10^{-15}\) & \(9.293\times 10^{-15}\) 
                                & \(7.79\times 10^{-15}\)  & \(9.196\times 10^{-15}\)\\                            
\midrule
\multirow{3}*{3}   & Iteration  & 137 & 49 & 56 & 49\\
                   & Time (s)   & 67.35 & 38.49 & 14.64 & 12.30 \\
                   & Residual   & \(9.486\times 10^{-15}\) & \(5.775\times 10^{-15}\) 
                                & \(9.601\times 10^{-15}\) & \(7.53\times 10^{-15}\) \\     
\midrule
\multirow{3}*{4}   & Iteration  & \(\bm{\geq 200}\) & 35 & 46 & 37 \\
                   & Time (s)   & \(\bm{299.5}\) & 33.18 & 27.07 & 22.03 \\
                   & Residual   & \(\bm{4.064\times 10^{-13}}\) & \(4.437\times 10^{-15}\) 
                                & \(2.946\times 10^{-15}\) & \(1.867\times 10^{-15}\)\\
\midrule
\multirow{3}*{5}    & Iteration & \(\bm{\geq 200}\) & 82 & 91 & 83\\
                    & Time (s)  & \(\bm{1672}\) & 1149 & 364.0 & 361.8 \\
                    & Residual  & \(\bm{1.536\times 10^{-12}}\) & \(7.763\times 10^{-15}\) 
                                & \(9.94\times 10^{-15}\) & \(7.435\times 10^{-15}\) \\
\bottomrule
\end{tabular}
\end{table}

\subsection{Real symmetric positive definite matrices}
In the following we compare several structured eigensolvers for the symplectic
eigenvalue problem.
For the structured LOBPCG algorithm, we transform the real symmetric positive
definite matrices to Bethe--Salpeter Hamiltonian matrices by the unitary
similarity described in Theorem~\ref{thm:equivalence}.
Unless otherwise specified, the preconditioner for all LOBPCG variants
is set to \(\Phi(A,0)\), implemented via the incomplete Cholesky factorization
with a drop tolerance of \(10^{-6}\).
Two typical symplectic eigensolvers---the restarted symplectic Lanczos
algorithm (SymplLanczos)~\cite{Amodio2006} and the Riemannian optimization
algorithm~\cite{SAGS2021}, are selected for comparison.
The SymplLanczos algorithm explicitly restarts after every \(\max\set{2l,50}\)
Lanczos steps when computing the \(l\) smallest symplectic eigenvalues.
The tolerances for convergence and for computing the coefficients required for
the restart initial vector are set to \(10^{-14} \) and \(10^{-12} \),
respectively;
see~\cite{Amodio2006} for details.

\subsubsection{Sparse symmetric positive definite matrices}
Five real symmetric definite matrices from the SuiteSparse Matrix Collection%
\footnote{URL: \url{https://sparse.tamu.edu/}.}
(formally, the University of Florida Sparse Matrix Collection~\cite{DH2011})
are selected as test matrices \(M\in\mathbb{R}^{2n\times 2n}\);
see Table~\ref{tab:spd-example}.
The time limit for each example is set to \(3{,}600\) seconds.

\begin{table}[!tb]
\centering
\caption{List of test examples of sparse symmetric positive definite matrices.}
\label{tab:spd-example}
\begin{tabular}{cccccc}
\toprule
Cases & Name & Size (\(n\)) &  nnz(\(M\)) & nnz(\(H\)) &\(\#\)Desired (\(l\))\\
\midrule
1     &\texttt{bcsstk21}           & \phantom{} 1,800    & \phantom{} 26,600 
                                   & \phantom{00} 27,800 & 18\\
2     &\texttt{fv1}                & \phantom{} 4,802    & \phantom{} 85,264 
                                   & \phantom{00} 87,016 & 48\\
3     &\texttt{crystm03}           & 12,384 & 583,770    & 1,288,140 & 100\\
4     &\texttt{apache1}            & 40,400 & 542,184    & \phantom{0} 562,320 & 100\\
5     &\texttt{shallow\_water2}    & 40,960 & 327,680    & \phantom{0} 660,480 & 100\\
\bottomrule
\end{tabular}
\end{table}

\begin{table}[!tb]
\centering
\caption{Performance comparison of LOBPCG variants for test examples listed
in Table~\ref{tab:spd-example}.
Data in boldface mean that the desired level of accuracy is not achieved
within the iteration limit or time limit.}
\label{tab:performance-spd}
\resizebox{\textwidth}{!}{
\begin{tabular}{lcccccc}
\toprule
\multirow{2}{*}{Method} & \multirow{2}{*}{Metric} & \multicolumn{5}{c}{Cases} \\
\cmidrule(lr){3-7}
& & 1 & 2 & 3 & 4 & 5 \\
\midrule
\multirow{2}{*}{SymplLanczos}
 & Time (s)  & 5.180 & 276.2 & 3381 & \(\bm{\geq 3600}\) & \(\bm{\geq 3600}\) \\
 & Residual  & \(7.000\times 10^{-15}\) & \(9.749\times 10^{-15}\) & \(9.087\times 10^{-15}\)
             & \(\bm{0.3728}\) & \(\bm{0.5535}\) \\
\midrule
\multirow{2}{*}{Riemannian}
 & Time (s)  & \(\bm{\geq 3600}\) & \(\bm{\geq 3600}\) & \(\bm{\geq 3600}\) 
             & \(\bm{\geq 3600}\) & \(\bm{\geq 3600}\) \\
 & Residual  & \(\bm{6.795\times 10^{-3}}\) & \(\bm{3.391\times 10^{-12}}\) & \(\bm{7.746\times 10^{-2}}\)
             & \(\bm{8.984\times 10^{-6}}\) & \(\bm{4.854\times 10^{-3}}\) \\
\midrule
\multirow{2}{*}{Algorithm~\ref{alg:LOBPCG-CIHL-OmegaIHL}}
 & Time (s)  & 4.948 & 41.01 & 553.6 & 1447 & 3256 \\
 & Residual  & \(9.247\times 10^{-15}\) & \(8.433\times 10^{-15}\) & \(7.837\times 10^{-15}\)
             & \(6.052\times 10^{-15}\) & \(9.841\times 10^{-15}\) \\
\bottomrule
\end{tabular}}
\end{table}

From Table~\ref{tab:performance-spd}, we observe that Algorithm~\ref{alg:LOBPCG-CIHL-OmegaIHL} 
demonstrates remarkable efficiency, successfully converging for all test cases listed in 
Table~\ref{tab:spd-example}.
The Riemannian algorithm fails to converge for all tested cases
within \(3{,}600\) seconds.
Notably, the LOBPCG-CIHL algorithm fails to solve \texttt{fv1}, \texttt{crystm03}, and
\texttt{shallow\_water2} to the desired accuracy within the prescribed iteration or time
limits, with its relative residuals stagnating near \(10^{-13}\).
This demonstrates the necessity of \(\Omega\)-orthogonalization in overcoming
accuracy barrier for challenging cases.
Therefore, in our subsequent numerical experiments, we shall focus on testing
Algorithm~\ref{alg:LOBPCG-CIHL-OmegaIHL} and discard other LOBPCG variants.

\subsubsection{Dense matrix with known symplectic eigenvalues}
\label{subsec:sym_exa1}
The second numerical experiment adheres to the test example in~\cite{SAGS2021}.
Let \(L\bigl(n/5,1.2,-\sqrt{n/5}\,\bigr)\in\mathcal S(2n)\) be the symplectic
Gauss transformation defined in~\cite{Fassbender2001}.
Let \(U_K\in\mathbb C^{n\times n}\) be a unitary matrix generated by
orthogonalization of a randomly generated complex matrix.
Denote a symmetric positive definite matrix \(M=Q\diag\{D,D\}Q\trans\), where
\[
D=\diag\set{1,\dotsc,n}, \qquad Q=KL\bigl(n/5,1.2,-\sqrt{n/5}\,\bigr),
\]
and
\[
K=\bmat{\Re(U_K) & \Im(U_K)\\-\Im(U_K) & \Re(U_K)}\in\mathcal{OS}(2n).
\]
We compute the \(20\) smallest positive eigenvalues and the corresponding
eigenvectors of the BSH matrix induced by \(M\) with
\(n\in\set{100,200,\dotsc,2000}\).
The maximum execution time for each example is set to \(300\) seconds.

The maximum relative errors and relative residuals for approximate eigenpairs
are shown in Figure~\ref{fig:ex1}.
The corresponding numerical behaviour is presented
in Table~\ref{tab:ex1-time}.
Algorithm~\ref{alg:LOBPCG-CIHL-OmegaIHL} outperforms both the SymplLanczos
and Riemannian algorithms, consistently achieving relative residual precision
below \(10^{-14}\) in the least amount of time.

\begin{figure}[!tb]
\centering
\begin{tabular}{cc}
\includegraphics[width=0.45\textwidth]{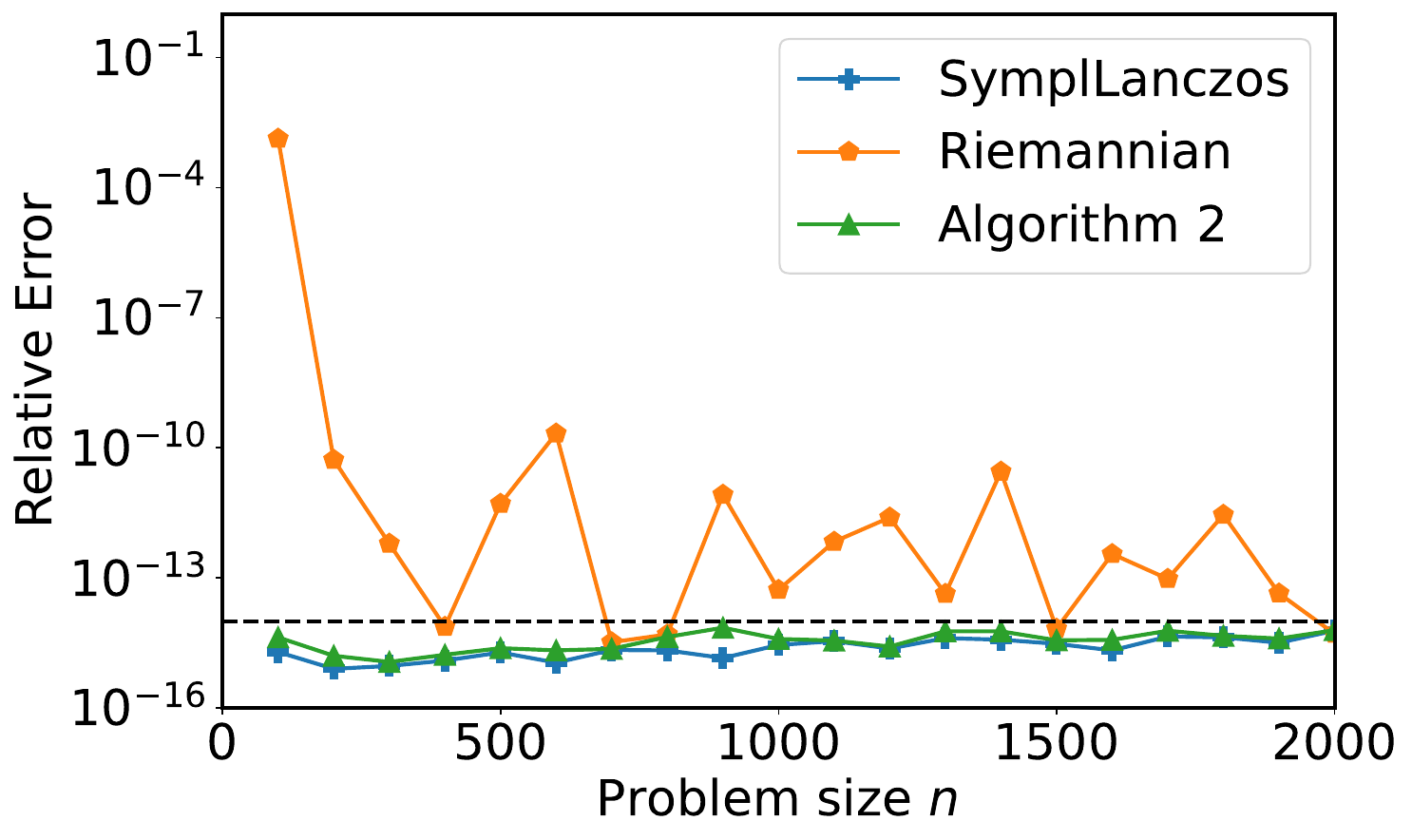} & 
\includegraphics[width=0.45\textwidth]{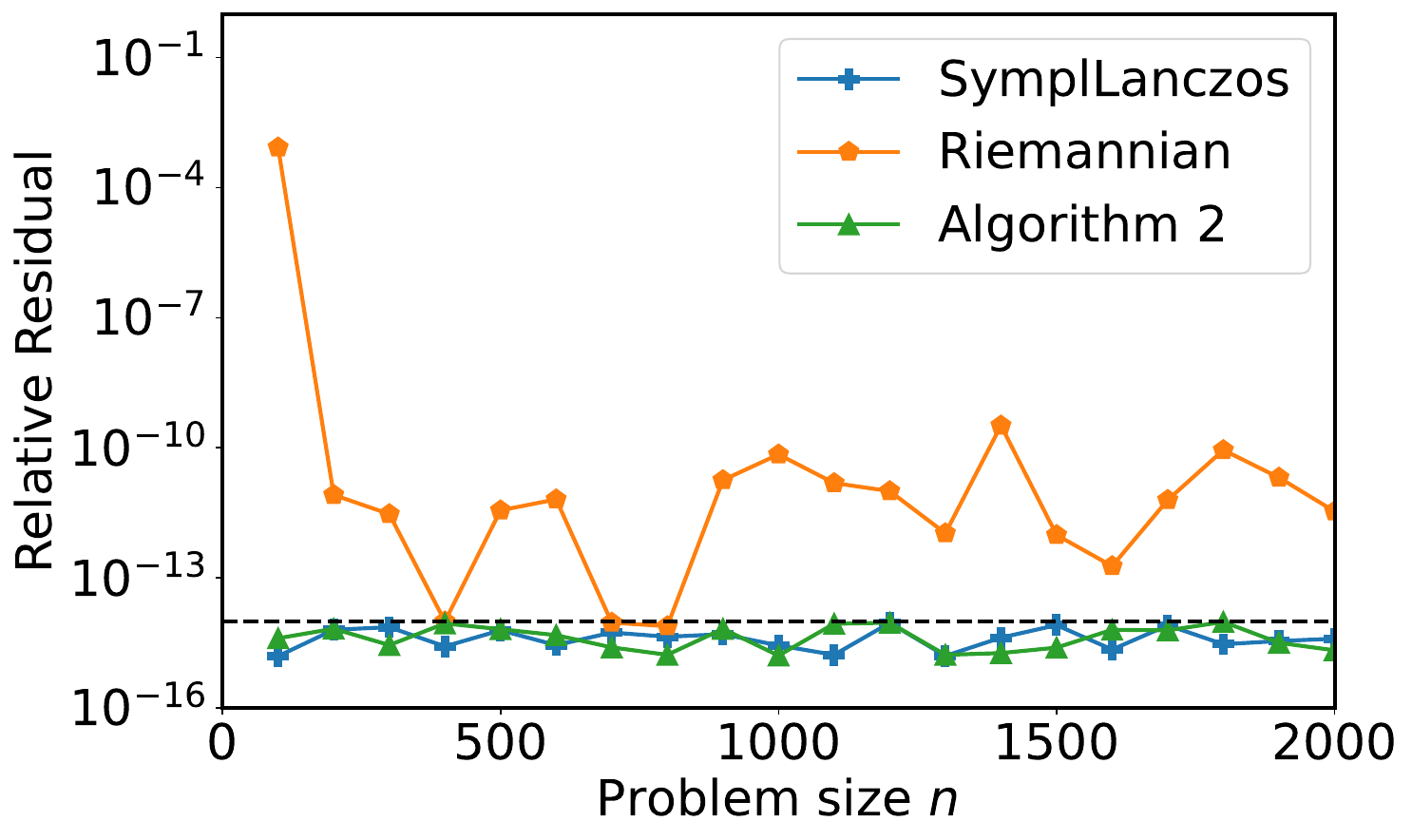}
\end{tabular}
\caption{The maximal relative errors of the computed eigenvalues (left) and their 
corresponding relative residuals (right)}
\label{fig:ex1}
\end{figure}

\begin{table}[!tb]
\centering
\small
\caption{Execution time (in seconds) and relative residual of different
methods (Section \ref{subsec:sym_exa1}).
Data in boldface mean that the desired level of accuracy (\(10^{-14}\)) is not
achieved.}
\label{tab:ex1-time}
\resizebox{\textwidth}{!}{
\begin{tabular}{lcccccc}
\toprule
\multirow{2}{*}{Method}& \multirow{2}{*}{Metric}& \multicolumn{5}{c}{\(n\)} \\
\cmidrule(lr){3-7}
 &   & 400 & 800 & 1200 & 1600 & 2000 \\
\midrule
\multirow{2}{*}{SymplLanczos}
& Time (s)  & 9.199 & 36.58 & 60.34 & 150.7 & 236.7\\
& Residual  & \(2.632\times 10^{-15}\) & \(4.390\times 10^{-15}\) 
            & \(9.189\times 10^{-15}\) & \(2.230\times 10^{-15}\) 
            & \(3.970\times 10^{-15}\)\\
\midrule
\multirow{2}{*}{Riemannian}
& Time (s)  & 69.89 & 94.15 & \(\bm{\geq300}\) & \(\bm{\geq300}\) & \(\bm{\geq300}\)\\
& Residual  & \(9.909\times 10^{-15}\) & \(7.707\times 10^{-15}\) 
            & \(\bm{1.004\times 10^{-11}}\) & \(\bm{1.871\times 10^{-13}}\) 
            & \(\bm{3.424\times 10^{-12}}\)\\
\midrule
\multirow{2}{*}{Algorithm~\ref{alg:LOBPCG-CIHL-OmegaIHL}}
& Time (s) & 2.588 & 10.55 & 25.96 & 65.72 & 129 \\
& Residual & \(8.918\times 10^{-15}\) & \(1.684\times 10^{-15}\) 
           & \(9.258\times 10^{-15}\) & \(6.369\times 10^{-15}\) 
           & \(2.137\times 10^{-15}\)\\
\bottomrule
\end{tabular}}
\end{table}

\subsubsection{Weakly damped gyroscopic system}
\label{subsec:sym_exa2}
The quadratic eigenvalue problem \((\lambda^2 N+\lambda G+K)x=0\)
generated in the stability analysis of the gyroscopic systems, is linearized and
discretized to the standard eigenvalue problem for the Hamiltonian matrix
\[
H_{\rm ham}=\frac14\bmat{-2GN^{-1}
& GN^{-1}G-K \\
4N^{-1} & -2N^{-1}G},
\]
where \(N\) and \(K\) are symmetric positive definite, and \(G\) is
skew-symmetric.
The elements of the matrix \(G\) typically have a much smaller magnitude
compared to \(K\).
Under such circumstances, \(M=J_nH_{\rm ham}\) is a symmetric positive
definite matrix.
The matrices \(N\), \(G\), and \(K\) are generated via an eigenfunction
discretization of a wire saw model (as described in~\cite{BFS2008}) with the
wire speed \(v=0.0306\).
We compute the five smallest symplectic eigenvalues and their corresponding
symplectic eigenvectors of \(M\).
The maximum execution time for each test is set to \(1{,}800\) seconds.

In this example, we utilize \(\Phi\bigl(\Diag(A),\Diag(\conj{B})\bigr)\) as
the preconditioner for Algorithm~\ref{alg:LOBPCG-CIHL-OmegaIHL}.
Table~\ref{tab:ex2-eig} presents the five smallest computed eigenvalues, along
with their corresponding residuals and execution time, for \(n=2000\) and
\(n=5000\).
Algorithm~\ref{alg:LOBPCG-CIHL-OmegaIHL} and the SymplLanczos algorithm successfully 
solve these ill-conditioned test problems, and the Riemannian algorithm fails to 
reduce the relative residual to the desired tolerance of \(10^{-14}\) within 
the prescribed time limit.
Furthermore, the absolute accuracy of the eigenvalues obtained by the Riemannian 
algorithm is relatively poor.

We further evaluate the numerical behaviour of the three methods in computing
the first \(50\) eigenvalues for \(n=5000\).
Within the prescribed time limit, the SymplLanczos and Riemannian algorithms
only achieve accuracies of \(2.564\times10^{-4}\) and \(2.134\times10^{-8}\),
respectively.
In comparison, Algorithm~\ref{alg:LOBPCG-CIHL-OmegaIHL} attains a
significantly higher accuracy of \(6.069\times10^{-15}\) in \(59.48\) seconds.

\begin{table}[!tb]
\centering
\caption{Comparison of different methods (Section~\ref{subsec:sym_exa2}).
Data in boldface mean that the desired level of accuracy (\(10^{-14}\)) is not
achieved.}
\label{tab:ex2-eig}
\renewcommand{\arraystretch}{1}
\begin{tabular}{cccc}
\toprule
Metric & SymplLanczos & Riemannian & Algorithm~\ref{alg:LOBPCG-CIHL-OmegaIHL} \\
\midrule
\multicolumn{4}{l}{Case \(1\): \(n=2000\), \(\cond(M)=9.88\times 10^6\)} \\
\midrule
\(\lambda_1\)   & 3.13865099189288 & 3.13865099189303 & 3.13865099189287 \\
\(\lambda_2\)   & 6.27730198378694 & 6.27730198379529 & 6.27730198378695 \\
\(\lambda_3\)   & 9.41595297568341 & 9.41595297570209 & 9.41595297568337 \\
\(\lambda_4\)   & 12.5546039675834 & 12.5546039678731 & 12.5546039675834 \\
\(\lambda_5\)   & 15.6932549594882 & 15.6932954338389 & 15.6932549594882 \\
Residual        & \(6.442\times 10^{-17}\) & \(\bm{6.940\times 10^{-10}}\)  
                & \(3.553\times 10^{-15}\) \\
Time (s)        & 182.8  & \(\bm{\geq 1800}\)  & 5.278 \\
\midrule
\multicolumn{4}{l}{Case \(2\): \(n=5000\), \(\cond(M)=6.18\times 10^7\)} \\
\midrule
\(\lambda_1\)   & 3.13865099189269 & 3.13874572363742 & 3.13865099189269 \\
\(\lambda_2\)   & 6.27730198378552 & 6.27748742406983 & 6.27730198378544 \\
\(\lambda_3\)   & 9.41595297567839 & 9.41658536766371 & 9.41595297567839 \\
\(\lambda_4\)   & 12.5546039675714 & 12.5556458218930 & 12.5546039675715 \\
\(\lambda_5\)   & 15.6932549594650 & 15.7193822016987 & 15.6932549594650 \\
Residual        & \(8.744\times 10^{-17}\) & \(\bm{2.750 \times 10^{-9}}\)  
                & \(2.759\times 10^{-15}\) \\
Time (s)        & \(907.8\)  & \(\bm{\geq 1800}\)  & 28.62\\
\bottomrule
\end{tabular}
\end{table}

\section{Conclusion and outlook}
\label{sec:conclusions}
In this paper, we present an adaptive structure-preserving LOBPCG algorithm
for computing a few smallest positive eigenpairs of the definite Bethe--Salpeter
eigenvalue problems.
The proposed algorithm employs an adaptive, multi-level orthogonalization
framework, significantly improving computational efficiency.
During the initial stage, the indefinite LOBPCG algorithm utilizes a selective 
reorthogonalization strategy to minimize computational overhead.
If convergence stagnation is detected in some scenarios, the algorithm 
framework adaptively switches to the \(\Omega\)-inner product setting to enhance
numerical stability.
The proposed algorithm is also well-suited for solving the symplectic
eigenvalue problem that is equivalent to the BSEP.
Numerical experiments have confirmed both the computational accuracy
and the efficiency of the algorithm.

The equivalence between the BSEP and the symplectic eigenvalue problem allows
for a natural extension of the LOBPCG algorithm from the \(C_n\)-inner product
to the \(J_n\)-inner product framework.
The induced LOBPCG solver can be applied to the computation of the symplectic
eigenvalues of an even-order, real symmetric positive definite matrix.

Though not reported in the experiments, we observe that the indefinite LOBPCG
algorithm can benefit from the \emph{shrink-and-expand technique} recently
proposed in~\cite{LMS2024} for some test cases.
However, the convergence curve sometimes exhibits oscillations in the
indefinite inner product setting, which limits the acceleration effect for
some test cases.
Further investigation is necessary to achieve more significant acceleration
effects.
This is planned as our future work.

\section*{Acknowledgments}
We thank Bin Gao, and Yuxin Ma for helpful discussions.
Additionally, we are grateful to Yuanfan Xiong, and Zhengbang Zhou for providing
the test matrix \texttt{PNR} presented in Table~\ref{tab:BSE-example}.
This work is partially supported by the National Natural Science Foundation
of China under grant No.~92370105.

\appendix
\section{Loss of orthogonality for the indefinite SVQB algorithm}
\label{sec:appendix-SVQB}
Let \(U\in\mathbb{C}^{2n\times 2p}\), and define the growth factor
\(\rho_U=\lVert U\rVert_2^2/\lVert M_U\rVert_2\).
We consider the explicit floating-point computation of the matrix product
\(M_U=U\herm C_nU\), which satisfies
\[
\hat M_U=M_U+\delta M_U,
\qquad
\lVert\delta M_U\rVert_2\leq c_1\macheps\lVert U\rVert_2^2
= c_1\rho_U\macheps\lVert M_U\rVert_2,
\]
where \(c_1\) is a constant depending on \(n\) and \(p\).

Moreover, we assume that the structured eigenvalue
problem~\eqref{eq:ortho-eig} is solved in a backward stable manner
so that the computed spectral decomposition satisfies
\[
\tilde M_U=\hat M_U+\delta\tilde M_U
=\hat FC_p\hat\Sigma\hat F\herm,\quad\quad
\lVert\delta\tilde M_U\rVert_2\leq c_2\macheps\lVert\hat M_U\rVert_2
\leq c_2\macheps\lVert M_U\rVert_2+O(\macheps^2),
\]
where \(\hat\Sigma=\diag\{\hat\Sigma_{+},\hat\Sigma_{+}\}\) is a positive
definite diagonal matrix, \(\hat F\) is numerically unitary (i.e.
\(\lVert\hat F\herm\hat F-I\rVert_2=O(\macheps)\)), and \(c_2\) is a constant
depending on \(n\) and \(p\).
This can be achieved by the structure-preserving algorithm provided
in~\cite{Shan2025}.
Let \(\sigma_i\)'s and \(\tilde\sigma_i\)'s be the singular values of \(M_U\)
and \(\tilde M_U\), respectively.
Then
\begin{equation}
\label{eq:S}
\lvert\tilde\sigma_i-\sigma_i\rvert\leq\lVert\tilde M_U-M_U\rVert_2
=\lVert\delta\tilde M_U+\delta M_U\rVert_2\leq c_3\macheps\lVert M_U\rVert_2,
\end{equation}
where \(c_3=c_1\rho_U+c_2\).
Thus, \(\tilde\sigma_{\min}\geq\sigma_{\min}-c_3\sigma_{\max}\macheps\).
As a result, we obtain
\begin{equation}
\label{eq:BarTheta}
\lVert\hat\Sigma^{-1}\rVert_2
\leq\frac{1}{\sigma_{\min}-c_3\sigma_{\max}\macheps}.
\end{equation}

Let \(U^{\{1\}}\) be the floating-point representation of the
\(C_n\)-orthonormal basis \(U^{Q}=U\hat F\hat\Sigma^{-1/2}\) generated by the
indefinite SVQB algorithm.
Then we have
\[
U^{\{1\}}=U^{Q}+\delta U^{Q}=U\hat F\hat\Sigma^{-1/2}+\delta U^Q,
\qquad
\lVert\delta U^Q\rVert_2
\leq c_4\macheps\lVert U\rVert_2\lVert\hat F\rVert_2
\lVert\hat\Sigma^{-1/2}\rVert_2,
\]
where \(c_4\) is a constant depending on \(n\) and \(p\).
Let
\[
\tilde c=\frac{\lVert U^Q\rVert_2}
{\lVert U\rVert_2\lVert\hat F\rVert_2\lVert\hat\Sigma^{-1/2}\rVert_2}
\in(0,1].
\]
We infer that
\[
\lVert U^{\{1\}}\rVert_2
\geq\lVert U^Q\rVert_2-\lVert\delta U^Q\rVert_2
\geq(\tilde c-c_4\macheps)\lVert U\rVert_2\lVert\lVert\hat F\rVert_2
\lVert\hat\Sigma^{-1/2}\rVert_2.
\]
Since \(c_4\macheps\) is very small, we assume that \(\tilde c-c_4\macheps>0\).
Then
\begin{align}
\lVert U^Q\rVert_2=\tilde c\lVert U\rVert_2\lVert\hat F\rVert_2\lVert\hat\Sigma^{-1/2}\rVert_2
\leq\frac{\tilde c}{\tilde c-c_4\macheps}\lVert U^{\{1\}}\rVert_2,\label{eq:error2}\\
\lVert\delta U^Q\rVert_2\leq c_4\macheps
\lVert U\rVert_2\lVert\hat F\rVert_2\lVert\hat\Sigma^{-1/2}\rVert_2
\leq\frac{c_4\macheps}{\tilde c-c_4\macheps}\lVert U^{\{1\}}\rVert_2.\label{eq:error3}
\end{align}
Based on \eqref{eq:S}, \eqref{eq:BarTheta}, \eqref{eq:error2}
and~\eqref{eq:error3}, the loss of orthogonality can be bounded as
\begin{align}
&\lVert (U^{\{1\}})\herm C_nU^{\{1\}}-C_p\rVert_2\nonumber\\
={}&\lVert(U^Q+\delta U^Q)\herm C_n(U^Q+\delta U^Q)-C_p\rVert_2\nonumber\\
\leq{}&\lVert(U\hat F\hat\Sigma^{-1/2})\herm C_n(U\hat F\hat\Sigma^{-1/2})
-C_p\rVert_2 + 2\lVert\delta U^Q\rVert_2\lVert U^Q\rVert_2
+ O(\lVert\delta U^Q\rVert_2^2)\nonumber\\
\leq{}&\lVert\hat\Sigma^{-1/2}\hat F\herm\tilde M_U\hat F\hat\Sigma^{-1/2}
-C_p\rVert_2 + \lVert\hat\Sigma^{-1/2}\hat F\herm(\delta\tilde M_U+\delta M_U)
\hat F\hat\Sigma^{-1/2}\rVert_2\nonumber\\
&+ 2\lVert\delta U^Q\rVert_2\lVert U^Q\rVert_2 + O(\lVert\delta U^Q\rVert_2^2)
\nonumber\\
\leq{}&\frac{c_3}{1-c_3\macheps\kappa(M_U)}\cdot\macheps\kappa(M_U)
+\frac{\tilde cc_4}{(\tilde c-c_4\macheps)^2}\cdot2\macheps
\lVert U^{\{1\}}\rVert_2^2.
\label{eq:ortho_bound}
\end{align}
Let
\[
\Delta=\frac{c_3}{1-c_3\macheps\kappa(M_U)}\kappa(M_U)
+\frac{2\tilde cc_4}{(\tilde c-c_4\macheps)^2}\lVert U^{\{1\}}\rVert_2^2,
\]
and assume that \(\Delta\cdot\macheps<1\).
Then we have
\[
\lVert (U^{\{1\}})\herm C_nU^{\{1\}}-C_p\rVert_2\leq\Delta\cdot\macheps,
\]
and
\[
1-\Delta\cdot\macheps\leq\sigma_i^{\{1\}}\leq1+\Delta\cdot\macheps,
\]
where \(\sigma_i^{\{1\}}\) is the \(i\)th largest singular value of
\((U^{\{1\}})\herm C_nU^{\{1\}}\).
Thus
\[
\kappa\bigl((U^{\{1\}})\herm C_nU^{\{1\}}\bigr)
\leq\frac{1+\Delta\cdot\macheps}{1-\Delta\cdot\macheps}.
\]
As long as \(\Delta\cdot\macheps\) is not too close to \(1\), it follows that
\(\kappa\bigl((U^{\{1\}})\herm C_nU^{\{1\}}\bigr)=O(1)\).
In this case, we perform the indefinite SVQB algorithm to \(U^{\{1\}}\) once more.
According to~\eqref{eq:ortho_bound}, this yields a new \(C_n\)-orthonormal
basis \(U^{\{2\}}\) that satisfies
\[
\bigl\lVert\bigl(U^{\{2\}}\bigr)\herm C_nU^{\{2\}}-C_p\bigr\rVert_2
\leq c_5\macheps\bigl(O(1)+2\lVert U^{\{2\}}\rVert_2^2\bigr)
=O(\macheps)\lVert U^{\{2\}}\rVert_2^2.
\]
Therefore, the orthogonality of the indefinite SVQB algorithm can be improved by one
step of reorthogonalization.

Finally, we remark that the growth factor \(\rho_U\) upon convergence is usually
not very large, at least for the leading block of \(U\).
When the LOBPCG algorithm converges, the approximate eigenvectors \(Z\),
satisfy \(Z\herm\Omega Z
=\diag\set{\lambda_1,\dotsc,\lambda_k,\lambda_1,\dotsc,\lambda_k}\) and
\(Z\herm C_nZ=C_k\).
Then the growth factor of the \(Z\) block is always bounded because
\[
\rho_Z=\frac{\lVert Z\rVert_2^2}{\lVert Z\herm C_nZ\rVert_2}
\leq\frac{\lambda_k}{\lambda_1}.
\]

\section{Accumulation of rounding errors on basis update}
\label{sec:appendix-basis}
Let \(U\in\mathbb C^{2n\times2k}\) be \(C_n\)-orthonormal, and \(V\in\mathbb
C^{2k\times2k}\) be \(C_k\)-orthonormal.
Then \(Z=UV\) is also \(C_n\)-orthonormal.
Suppose that \(\hat U\) and \(\hat V\), respectively, are the computed results
of \(U\) and \(V\) in floating-point arithmetic, satisfying
\[
\lVert\hat U\herm C_n\hat U-C_k\rVert_2\leq O(\macheps)\lVert\hat U\rVert_2^2,
\qquad
\lVert\hat V\herm C_k\hat V-C_k\rVert_2\leq O(\macheps)\lVert\hat V\rVert_2^2.
\]
Ideally, we would like to formulate \(\hat Z\approx Z\) such that
\begin{equation}
\label{eq:ortho-loss-Z}
\lVert\hat Z\herm C_n\hat Z-C_k\rVert_2
\leq O(\macheps)\lVert\hat Z\rVert_2^2.
\end{equation}
Unfortunately, even if \(\hat Z=\hat U\hat V\) is computed exactly, the loss
of \(C_n\)-orthogonality is bounded by
\begin{equation}
\label{eq:ortho-loss-UV}
\lVert\hat Z\herm C_n\hat Z-C_k\rVert_2
\leq O(\macheps)\lVert\hat U\rVert_2^2\lVert\hat V\rVert_2^2.
\end{equation}
Note that \(\lVert V\rVert_2\geq1\).
This implies that a multiplicative update of an \(C_n\)-orthonormal basis
(i.e., \(U\mapsto UV\)) almost always causes a larger accumulation of rounding
errors in floating-point arithmetic.
However, the right-hand side of~\eqref{eq:ortho-loss-UV} can potentially be
much larger than that of~\eqref{eq:ortho-loss-Z}, especially when
\(\rho_{\hat Z}<\rho_{\hat U}\) (or, equivalently,
\(\lVert\hat Z\rVert_2<\lVert\hat U\rVert_2\)).
In order to produce a better \(C_n\)-orthonormal basis, it is recommended to
explicitly perform one step of \(C_n\)-orthogonalization on \(\hat Z\), so
that the loss of \(C_n\)-orthogonality becomes
\[
\lVert\hat Z_{\mathrm{new}}\herm C_n\hat Z_{\mathrm{new}}-C_k\rVert_2
\leq O(\macheps)\lVert\hat Z_{\mathrm{new}}\rVert_2^2,
\]
which is comparable to the ideal bound in~\eqref{eq:ortho-loss-Z}.

\addcontentsline{toc}{section}{References}

\end{document}